\documentclass[journal,onecolumn]{IEEEtran}

\usepackage{nicematrix}

\usepackage{amsmath} 
\usepackage{amssymb}  
\usepackage{amsthm} 
\usepackage{amsfonts}
\usepackage{graphicx}
\usepackage{xcolor}
\usepackage{mathtools}
\usepackage{subcaption}
\usepackage{enumitem}
\usepackage{lipsum}
\usepackage{mathtools}
\usepackage{cuted}
\usepackage{cite}

\usepackage[hidelinks]{hyperref}

\usepackage{mathrsfs} 
\newtheorem{remark}{Remark}
\newtheorem{theorem}{Theorem}
\newtheorem{lemma}{Lemma}

\newtheorem{definition}{Definition}
\newtheorem{corollary}{Corollary}
\newtheorem{proposition}{Proposition}

\usepackage{pgf,tikz}
\usepackage[utf8]{inputenc}
\usepackage{pgfplots} 
\usepackage{pgfgantt}
\usepackage{pdflscape}
\pgfplotsset{compat=newest} 
\pgfplotsset{plot coordinates/math parser=false}

\newcommand\scalemath[2]{\scalebox{#1}{\mbox{\ensuremath{\displaystyle #2}}}} 

\usepackage{verbatim}
\def\q{2.1}
\def\h{1.8} 
\def\l{3.45} 
\def\ll{3} 
\def\qqq{2.6} 
\def\qq{2.55} 

\begin{document}
\title{\LARGE \bf
A Polarized Temporal Network Model to Study the Spread of Recurrent Epidemic Diseases in a Partially Vaccinated Population
}

\author{Kathinka Frieswijk, Lorenzo Zino, and Ming Cao
\thanks{Some preliminary results have appeared in~\cite{frieswijk2022ecc}, in the form of a conference paper. The authors are with the Engineering and Technology Institute Groningen, University of Groningen, 9747 AG Groningen, The Netherlands, e-mail: \texttt{ \{k.frieswijk,lorenzo.zino,m.cao\}\@rug.nl}. The work was partially supported by the European Research Council (ERC--CoG--771687).
}%
}

\maketitle

\begin{abstract}
\noindent Motivated by massive outbreaks of COVID-19 that occurred even in populations with high vaccine uptake, we propose a novel multi-population temporal network model for the spread of recurrent epidemic diseases. We study the effect of human behavior, testing, and vaccination campaigns on the control of local outbreaks and infection prevalence. Our modeling framework decouples the vaccine effectiveness in protecting against transmission and the development of severe symptoms. Furthermore, the framework accounts for the polarizing effect of the decision to vaccinate and captures homophily, i.e., the tendency of people to interact with like-minded individuals. By means of a mean-field approach, we analytically derive the epidemic threshold. Our theoretical results suggest that, while vaccination campaigns reduce pressure on hospitals, they might facilitate resurgent outbreaks, highlighting the key role that testing campaigns may have in eradicating the disease. Numerical simulations are then employed to confirm and extend our theoretical findings to more realistic scenarios. Our numerical and analytical results agree that vaccination is not sufficient to achieve full eradication, without employing massive testing campaigns or relying on the population's responsibility. Furthermore, we show that homophily plays a critical role in the control of local outbreaks, highlighting the peril of a polarized network structure. \end{abstract}

\section{Introduction}\label{sec:introduction}
\IEEEPARstart{I}{n}  response to the COVID-19 pandemic that emerged in Wuhan, China, in December 2019, an unprecedented effort has been made toward developing a vaccine in a ground-breaking record time~\cite{whovaccintracker}. Although the current COVID-19 vaccines are highly effective against the development of severe symptoms --- thereby reducing the number of deaths and pressure on hospitals --- they provide limited protection against transmission, especially concerning the new variants that became dominant in 2021--22~\cite{tang2022respiratory,siddle2022transmission}. As massive outbreaks have occurred in populations with a high vaccination coverage~\cite{siddle2022transmission}, it begs the following question. \emph{Is it possible to eradicate the pandemic by solely relying on vaccination campaigns? If not, could the implementation of testing campaigns and human behavior be key to achieving such a desirable goal? }

To shed light on the dynamics of the spread of an epidemic disease, it is common to employ mathematical modeling. In particular, epidemic models on networks have emerged as a powerful framework to incorporate the human-to-human interaction and mobility patterns through which diseases are transmitted~\cite{RevModPhys.87.925,nowzari2016analysis,dyncontrolmei,Pare2020review,zinoreview}. 
Among the reasons for their success, we mention that network epidemic models are amenable to developing tractable and analytically rigorous mathematical frameworks to study the spread of epidemic diseases~\cite{qu2017sis,fagnani2019sis,prasse2020epidemics} and design policies to control them~\cite{drakopoulos2016cure}. Moreover, they are amenable to several extensions toward incorporating further real-world features, such as awareness~\cite{sahneh2017contact}. In particular, network epidemic models have been successfully extended and employed to study the effect of vaccination campaigns~\cite{preciado2018vaccine,Tetteh2021,Bongiorno2022}. Recently, network models have been applied to the COVID-19 pandemic, to support governments in defining the most effective control strategies~\cite{parino2021,Carli2020,DellaRossa2020,Chen2020covid}.

Here, motivated by the questions we formulated above, we propose a network model to address the spread of recurrent epidemic diseases, which include, for example, infections caused by viruses that do not convey permanent immunity (e.g., COVID-19) or fast-mutating viruses (e.g., influenza viruses). 
Our model extends the standard susceptible--infected--susceptible (SIS) model by separating mildly symptomatic and severely symptomatic infected individuals into two distinct compartments, where the latter are isolated. The proposed framework includes a responsibility level, reflecting the extent to which mildly symptomatic individuals do not overlook their symptoms and maintain physical distance from others to protect them by avoiding transmission. 
We consider two measures to control the spreading of the infection. First, we consider \textit{vaccination}, where we assume that a fixed part of the population is vaccinated. The effect of vaccination is modeled through the parsimonious addition of two distinct parameters to elegantly decouple the effects of the vaccine on reducing contagion and on protection against severe symptoms. Secondly, we consider the implementation of \textit{testing campaigns}, employed to identify and isolate mildly symptomatic infected individuals. In our network model, we do not rely on any time-scale separation assumption. Hence, we consider a co-evolution of the network of human-to-human interactions and the epidemic spreading. To this aim, we develop our epidemic model using a network formation process inspired by continuous-time activity-driven networks~\cite{ADN,zino1}. 

Despite the consensus among the scientific community that vaccines that are approved by public health authorities are effective and safe, the debate on whether to vaccinate or not is very polarizing, and individuals with similar attitudes toward vaccination prefer to interact with one another~\cite{Bessi2016,schmidt2018polarization,monsted2022characterizing}. Here, we investigate the effect that this phenomenon of \textit{homophily} has on the spread of epidemic diseases by developing a mathematical model that encapsulates such a phenomenon. Specifically, we expand our preliminary effort in~\cite{frieswijk2022ecc} by embedding the proposed model in a multi-population network scenario that accounts for the polarizing effect of the vaccination decision, capturing the tendency of individuals to establish more interactions with like-minded people.

Employing a mean-field approach on a large population~\cite{virusspread,9089218}, we perform a theoretical analysis of our network model and derive a closed-form expression for the epidemic threshold. Such an analytical expression allows for shining light on the roles that network polarization, individuals' responsibility, and vaccination play in controlling local outbreaks. Next, inspired by~\cite{Nadini2020} --- which incorporates preferential contacts in the modeling framework by considering activity-driven networks superimposed on a static backbone network --- we perform numerical simulations while restricting interactions to an underlying social network structure. Here, the parameter values are calibrated on the COVID-19 pandemic. 

Our theoretical results suggest that, while vaccination is beneficial in reducing the number of deaths, its role in controlling local outbreaks is nontrivial, and depends on individuals' responsibility levels and the characteristics of both the vaccine and the infection. Hence, vaccination could in some cases act as a double-edged sword, hindering the complete eradication of the disease, for which the implementation of massive testing campaigns becomes necessary. Our simulations provide further insights into the role of human behavior. Notably, they suggest that responsibility is key; for low responsibility levels, it is impossible to eradicate the infection without the employment of massive testing campaigns. Furthermore, both our theoretical and our numerical results show that a high degree of homophily facilitates the spreading of a disease. This underlines the peril of a polarized network, in which clusters of individuals disregard vaccines and the use of protective measures.

The rest of the paper is organized as follows. In Section~\ref{sec:preliminaries}, we introduce the notation and some mathematical preliminaries. In Section~\ref{sec:model}, we illustrate our modeling framework. Section~\ref{section:mainresults} presents the analysis of the model and our main theoretical results. Section~\ref{sec:simulations} discusses our numerical findings. Section~\ref{Sec:conclusion} concludes the paper and outlines future research.

\section{Notation and preliminaries}\label{sec:preliminaries}

\noindent In this section, we gather some notational conventions and we present some key definitions and properties of stochastic processes, which will be used in the rest of the paper. More details on stochastic processes can be found in~\cite{Bailey1990,levin2006book}.

The set of non-negative real, strictly positive real, and non-negative integer numbers is denoted by $\mathbb{R}_{\ge 0}$, $\mathbb{R}_{> 0}$, and $\mathbb{Z}_{\ge 0}$, respectively. Given a function $x(t)$ with $t \in \mathbb{R}_{\ge 0}$, we define $x(t^+) = \lim_{s \searrow t} x(s)$, and $x(t^-) = \lim_{s \nearrow t} x(s)$. Given an event $E$, we denote by $\mathbb P[E]$ its probability. Given a random variable $X$, we denote by $\mathbb E[X]$ its expected value.

\begin{definition}\label{def:poisson}
A {Poisson clock} with (possibly time-varying) rate $\gamma(t)$ is a continuous-time stochastic process, represented by its counting process $N(t)\in\mathbb Z_{\geq 0}$. Specifically, $N(t)$ is a non-decreasing function that satisfies
\begin{equation}\label{eq:clock}
 \mathbb P[N(t+\Delta t)-N(t)=1]=\int_{t}^{t+\Delta t}\gamma(t)\, \mathrm{d}t+o(\Delta t)\,,
\end{equation}
for $\Delta t\in\mathbb R_{>0}$, where the Landau little-o notation $o(\Delta t)$ is associated with the limit $\Delta t\searrow 0$; hence, $\lim_{\Delta t\searrow 0} \mathbb P[N(t+\Delta t)-N(t)=1]/\Delta t=\gamma(t)$. If $N(t)$ has an increment at time $t\in\mathbb R_{\geq 0}$, we say that the clock ticks at time $t$.
\end{definition}

\begin{proposition}\label{prop:poisson_minimum}
The following two properties hold:
\begin{enumerate}[label=\roman*)]
 \item Let $E$ be an event triggered by the first tick of a set of independent Poisson clocks with rates $\gamma_1(t),\dots,\gamma_\ell(t)$. Then, event $E$ can be equivalently described as triggered by a Poisson clock with rate $\gamma_E(t):=\sum_{h=1}^\ell\gamma_h(t)$.
 \item Let $E$ be an event that occurs when a Poisson clock with rate $\gamma(t)$ ticks, with probability $p\in[0,1]$, independent of the Poisson clock. Then, event $E$ can be equivalently described as triggered by a Poisson clock with rate equal to $\gamma_E:=p\gamma$.
\end{enumerate}
\end{proposition}

\begin{definition}\label{def:markov}
A continuous-time stochastic process $X(t)$ with the state space $\mathcal A$ is a {Markov process} if, for any states $h,k\in\mathcal A$, the transition from $h$ to $k$ is triggered by a Poisson clock with rate $q_{hk}(t)$, independent of the others. 
The transition rates are gathered in the transition rate matrix $Q(t) \in \mathbb{R}^{|\mathcal A| \times |\mathcal A|}$.
\end{definition}

\section{Model}\label{sec:model}

\noindent We extend the traditional network SIS model~\cite{zinoreview} by separating infected individuals into two distinct compartments: i) \textit{infectious} individuals, who are untested and mildly symptomatic, and ii) \textit{quarantined} individuals, who are isolated as a result of a positive test result, or because they are severely symptomatic. In the rest of this section, we will present all the details of our multi-population epidemic model.

\subsection{Multi-population}

We consider a population of $n$ individuals $\mathcal{V} = \{1, \hdots, n\}$, connected through a time-varying undirected network $\mathcal G(t):=(\mathcal V,\mathcal E(t))$, which evolves in continuous time, $t\in\mathbb R_{\geq 0}$. Each individual $j\in\mathcal V$ is characterized by their health state $X_j(t) \in \{ \mathrm{S}, \mathrm{I}, \mathrm{Q}\}$, denoting susceptible, infectious, and quarantined infected individuals, respectively. 

The population is partitioned into two subpopulations according to vaccination status: i) a fully \emph{vaccinated} subpopulation of size $n_{\mathrm{v}}\in \{1,\dots,n-1\}$, and ii) a \emph{non-vaccinated} subpopulation of size $n-n_{\mathrm{v}}$. Without any loss of generality, we assume that individuals $\mathcal V_{\mathrm{v}}:=\{1,\dots,n_{\mathrm{v}}\}$ belong to the first subpopulation, while individuals $\mathcal V_{\mathrm{n}}:=\{n_{\mathrm{v}}+1,\dots,n\}$ belong to the second one. Let $v := \frac{n_{\mathrm{v}}}{n} \in [0,1]$ be the \emph{vaccination coverage} of the population. Thus, the couple $(\mathcal V,v)$ fully characterizes the multi-population model.

The disease is transmitted via pairwise interactions at close distance, which evolve at a timescale that is comparable with the epidemic spreading. Such an interaction is henceforth named \textit{contact}. These moments of contact are modeled via a time-varying undirected network $\mathcal G(t):=(\mathcal V,\mathcal E(t))$, with $t\in \mathbb{R}_{\ge 0}$, where the time-varying edge set $\mathcal{E}(t)$ denotes the interactions between individuals at time $t\in\mathbb{R}_{\ge 0}$. Specifically, if $(j,k)\in\mathcal{E}(t)$, then individuals $i$ and $j$ interact in close proximity (that is, they have contact) at time $t$. As in the case of many infectious diseases, the network formation process and the disease transmission process not only evolve at comparable timescales, but they are also deeply intertwined, as we detail in the following.

\subsection{Network formation}\label{sec:formation}

Interactions are generated in a stochastic fashion, inspired by activity-driven networks in continuous time~\cite{zino2}. To each individual $j \in \mathcal{V}$, we assign a Poisson clock with unit rate, which ticks independently of the other clocks. When the clock ticks, the individual initiates an interaction. To capture an individual's preferred engagement with like-minded individuals from their own subpopulation, we introduce a parameter $\theta \in [0,1)$. Specifically, if the clock associated with individual $j\in\mathcal V$ ticks at time $t\in\mathbb R_{\geq 0}$, then $j$ activates and has an interaction with another individual $k$, which is chosen according to a probabilistic rule: with probability $\theta$, $k$ is selected uniformly at random from $j$'s own subpopulation; whereas with probability $1-\theta$, $k$ is chosen uniformly at random from the entire population.

Whether the individuals interact at close proximity (so the individuals are considered to be in contact) is dependent on the health state and responsibility levels of the individuals involved. In particular, we assume that individuals $j$ who are quarantined ($X_j(t)=\mathrm{Q}$) do not have interactions in close proximity with other individuals. Infectious individuals with mild symptoms ($j:X_j(t)=\mathrm{I}$), instead, are able to be in close proximity with others. Whether they have contact with others depends on the individuals' level of responsibility. Specifically, let $\sigma_j \in [0,1]$ denote the responsibility of individual $j \in \mathcal{V}$. If an infected individual $j$ is mildly symptomatic ($X_j(t)=\mathrm{I}$), they refrain from having contact with probability $\sigma_j$; while with probability $1-\sigma_j$, they disregard their symptoms and do not maintain physical distance while interacting. We assume that the decision to maintain physical distance or not is made independently of the past and of other individuals. For the sake of simplicity, we assume in the following that the responsibility level within a subpopulation is homogeneous, and let $\sigma_j=\sigma_{\mathrm{v}}\in [0,1]$, for all $j\in\mathcal V_{\mathrm{v}}$, and $\sigma_i=\sigma_{\mathrm{n}}\in [0,1]$, for all $i\in\mathcal V_{\mathrm{v}}$, which denote the responsibility level of vaccinated and non-vaccinated individuals, respectively. Nonetheless, heterogeneous responsibility levels across the population can be easily introduced within our modeling framework by considering non-trivial distributions of the responsibility $\sigma_j$ across the population. 

To summarize, if a susceptible individual $j$ ($X_j(t)=\mathrm{S}$) 
activates and selects a mildly symptomatic infectious individual 
$k$ ($X_k(t)=\mathrm{I}$), then they have contact with probability equal to $1-\sigma_k$; whereas if $k$ is susceptible ($X_k(t^-)=\mathrm{S}$), then they always have contact. If a mildly symptomatic infectious individual $j$ ($X_j(t^-)=\mathrm{I}$) activates and interacts with a susceptible individual $k$ ($X_k(t^-)=\mathrm{S}$), then they have contact with probability $1-\sigma_j$; if $k$ is mildly symptomatic too ($X_k(t^-)=\mathrm{I}$), then they have contact with probability $(1-\sigma_j)(1-\sigma_k)$ (i.e., if they both ignore the symptoms). Finally, we recall that quarantined individuals ($j:X_j(t^-)= \mathrm{Q}$) are assumed to always maintain distance and thus they do not establish any contact. If individuals $j$ and $k$ have contact at time $t$, then the ephemeral edge $(j,k)$ is included in the set $\mathcal{E}(t)$, and instantaneously removed from the edge set $\mathcal{E}(t^+)$. 

\subsection{Disease transmission and control}\label{Epidemicmodel}

The evolution of the health state of each individual $j \in \mathcal{V}$ is governed by the following two natural mechanisms (contagion and recovery), and by free testing campaigns, where the intensity of the latter is interpreted as a control input.

{\bf Contagion.} The infection is transmitted via close contact. Furthermore, we assume that vaccination reduces the risk of becoming infected and, if infected, it reduces the risk of developing severe illness. To model these effects, we introduce two parameters: $\gamma_{\mathrm{t}} \in [0,1]$ and $\gamma_{\mathrm{q}} \in [0,1]$, respectively. Specifically, the contagion process acts as follows.

If a susceptible individual $j$ $(X_j(t^-)=\mathrm{S}$) has contact with a mildly symptomatic infectious individual $k$ $(X_k(t^-) =\mathrm{I})$ at time $t$, so $(j,k)\in\mathcal E(t)$, then $j$ becomes infected with \textit{per-contact infection probability} $\lambda \in [0,1]$ if $j$ is not vaccinated. Such a probability is reduced to $\lambda(1-\gamma_{\mathrm{t}})$ if $j$ is vaccinated. 
If the infection is transmitted, then $j$ will either move to health state $\mathrm{I}$ or to $\mathrm{Q}$. Specifically, the individual will become severely symptomatic $(X_j(t^+) =\mathrm{Q})$ with probability $p_{\mathrm{q}}\in [0,1]$ if $j$ is not vaccinated, while this probability is reduced to $p_{\mathrm{q}}(1-\gamma_{\mathrm{q}} )\in [0,1]$ if $j$ is vaccinated. Otherwise, the individual becomes infectious with mild symptoms ($X_j(t^+) =\mathrm{I}$).

{\bf Recovery.} An infected individual $j\in \mathcal{V}$ with $X_j(t^-) \in\{ \mathrm{I}, \mathrm{Q} \}$ spontaneously recovers when a Poisson clock with rate $\beta \in \mathbb{R}_{>0}$ ticks, thereby becoming susceptible again to the disease $(X_j(t^+)=\mathrm{S})$.

{\bf Testing.} Free testing campaigns are implemented to induce mildly symptomatic infectious individuals to get tested. To model the effect of free testing, we employ a Poisson clock with rate $\tau \in \mathbb{R}_{\ge 0} $, representing the rate of testing. Hence, an infectious individual $j$ with mild symptoms $(X_j(t^-)=\mathrm{I})$ receives a diagnosis when a Poisson clock with rate $\tau$ ticks. After being diagnosed, $j$ goes in quarantine $(X_j(t^+)=\mathrm{Q})$ and maintains physical distance with respect to other individuals until recovery takes place.

\begin{figure}
\vspace{8pt}
\centering
 \begin{tikzpicture} \node[draw=red, fill=red!10,circle, ultra thick,,minimum size=0.8cm] (I) at (3,1.5) {\large $\mathrm{I}$};
\node[draw=green, fill=green!10,circle, ultra thick,,minimum size=0.8cm] (S) at (0,1.50) {\large $\mathrm{S}$};
\node[draw=blue, fill=blue!10,circle, ultra thick,,minimum size=0.8cm] (Q) at (6,1.50) {\large $\mathrm{Q}$};

\path [->,>=latex,thick] (S) edge[bend left =25] node [below=-2pt] {{$\kappa_{{\alpha},j}$}} (I);
\path [->,>=latex,thick] (S) edge[bend left =25] node [above=-2pt] {{$\nu_{{\alpha},j}$}} (Q);
\path [->,>=latex,thick] (I) edge node [above=-2pt] {{$\tau$}} (Q);
\path [->,>=latex,thick] (I) edge[bend left =25] node [above=-2pt] {{$\beta$}} (S);
\path [->,>=latex,thick] (Q) edge[bend left =25] node [above=-3pt] {{$\beta$}} (S);
\end{tikzpicture}
\caption{ State transitions of the epidemic model for individual $j\in\mathcal V$, who belongs to subpopulation $\mathcal{V}_{\alpha}$ with $\alpha\in\{\mathrm{n},\mathrm{v}\}$. }
 \label{model} 
 \vspace{2pt}
\end{figure}
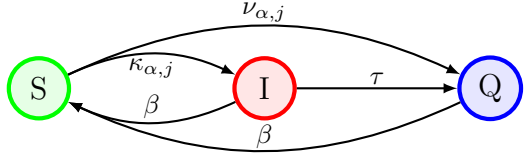

\noindent All the parameters are summarized in Table~\ref{tab:parameter}.

Before presenting the analysis of the dynamics and our main theoretical results, we would like to comment that our modeling framework is amenable to several extensions and generalizations, toward even more realistic settings. 

\begin{remark}\label{rem:npi}
Our model fits the implementation of non-pharmaceutical interventions by introducing a parameter $\eta\in[0,1]$ that captures their effectiveness in reducing the per-contact infection probability, similar to~\cite{frieswijk2022ecc,parino2021}. All our analytical findings are readily extended to this scenario by replacing all the occurrences of $\lambda$ with $(1-\eta)\lambda$.
\end{remark}

\begin{remark}\label{rem:network}
Our modeling framework can incorporate an underlying time-invariant network $\mathcal H=(\mathcal V,\mathcal F)$, which constrains the interactions that $j\in \mathcal{V}$ can have, similar to~\cite{karsai2014backbone,surano2019backbone,Nadini2020}. Specifically, rather than letting an individual $j$ select someone uniformly at random from the network or the own subpopulation, one can instead restrict the selection (uniform, at random) to the neighbor set of $j$, i.e., $\mathcal N_j:=\{k:(j,k)\in\mathcal F\}$, or to individuals $k \in \mathcal{N}_j$ who belong to the same subpopulation, respectively. 
\end{remark}

\begin{remark}\label{rem:sigma}
The responsibility level $\sigma_j$, $j \in \mathcal{V}$, which is assumed to be constant here, may in general be time-varying, influenced by the epidemic spreading and by co-evolving opinion formation processes. Opinion-dynamics~\cite{Xuan2020opinion,She2022opinion} or game-theoretic update mechanisms~\cite{ye2021game,hota2022,frieswijk2022mean} could be directly incorporated within our framework to model such a temporal evolution.
\end{remark}

While the extension in Remark~\ref{rem:npi} is straightforward, the other generalizations may complicate or hinder the analytical tractability of the system. For this reason, in the next sections, we will perform the theoretical analysis of the system while adhering to the original model implementation. The generalization discussed in Remark~\ref{rem:network} will subsequently be investigated in Section~\ref{sec:simulations}, via numerical simulations.

\subsection{Dynamics}\label{sec:dynamics}

All the mechanisms described in Section~\ref{Epidemicmodel} are induced by independent Poisson clocks, which implies that the evolution of the $n$-dimensional state 
 $X(t) : = \left[ X_1(t), X_2(t) , \dots, X_n(t) \right]\in\{ \mathrm{S},\mathrm{I},\mathrm{Q}\}^{n}$ 
is governed by a continuous time Markov process~\cite{levin2006book}.

Each individual is able to experience five distinct state transitions, illustrated in Fig.~\ref{model}, which are triggered by the processes of contagion, testing, and recovery. The two transitions triggered by recovery (from $\mathrm{I}$ and $\mathrm{Q}$ to $\mathrm{S}$) and the one triggered by testing (from $\mathrm{I}$ to $\mathrm{Q}$) solely involve spontaneous mechanisms. Hence, these three transition rates are simply given by the rates of the corresponding Poisson processes that trigger them. On the contrary, the two transitions induced by contagion (from $\mathrm{S}$ to $\mathrm{I}$ and $\mathrm{Q}$) depend on the vaccination status of the individual, and are also related to interactions between individuals and the health states of the others. Hence, they have a non-trivial time-varying expression, which is computed in the following. For the sake of readability, we omit to stress that the rates and the health states of individuals are functions of time. 


\begin{table}
\centering
\caption{Model and control parameters.}\label{tab:parameter}
\begin{tabular}{r| l}
$n\in\mathbb N$& population size\\
$t\in\mathbb R_{\geq 0}$& time\\
$X_j(t)$& health state of individual $j\in\mathcal V$ at time $t$\\
$(\mathcal V,\mathcal E(t))$& contact network (interactions in close proximity)\\
$v\in(0,1)$& vaccination coverage\\
$\mathcal V_{\mathrm{v}}\subset \mathcal V$& vaccinated subpopulation\\
$\mathcal V_{\mathrm{n}}\subset \mathcal V$& non-vaccinated subpopulation\\
$\sigma_{\mathrm{v}}\in[0,1]$ &responsibility level of vaccinated individuals\\
$\sigma_{\mathrm{n}}\in[0,1]$ &responsibility level of non-vaccinated individuals\\
$\lambda\in[0,1]$& per-contact infection probability\\
$p_{\mathrm{q}}\in[0,1]$& probability of severe illness\\
$\beta\in\mathbb R_{>0}$& recovery rate\\
$\gamma_{\mathrm{t}}\in[0,1]$&effectiveness of vaccine against transmission\\
$\gamma_{\mathrm{q}}\in[0,1]$&effectiveness of vaccine against severe illness\\
$\tau \in\mathbb R_{\geq 0}$&testing rate\\
$\theta \in [0,1)$ & homophily level with respect to interactions
\end{tabular}
\end{table}

\begin{proposition}\label{prop:contagionrates}
 A non-vaccinated individual $j$ with $X_j(t^-) = \mathrm{S}$ becomes infectious ($X_j(t^+) = \mathrm{I}$) according to a Poisson clock with rate
 \begin{align}\label{kappauv}
 \kappa_{\mathrm{n},j} := & 2\lambda \left(1-p_{\mathrm{q}}\right) \Bigg[(1-\sigma_{\mathrm{n}}) \left(\frac{\theta}{n(1-v)-1} + \frac{1-\theta}{n-1} \right) \sum\limits_{\substack{k\in\mathcal V_{\mathrm{n}}: X_k= \mathrm{I}}} 1+(1-\sigma_{\mathrm{v}}) \frac{1-\theta}{n-1}\sum\limits_{\substack{k\in\mathcal V_{\mathrm{v}}: X_k= \mathrm{I}} } 1\Bigg],
\end{align}
while they become quarantined ($X_j(t^+) = \mathrm{Q}$) with rate 
 \begin{align}\label{nuuv}
 \nu_{\mathrm{n},j} := & 2 \lambda p_{\mathrm{q}}\Bigg[(1-\sigma_{\mathrm{n}}) \left(\frac{\theta}{n(1-v)-1} + \frac{1-\theta}{n-1} \right) \sum\limits_{\substack{k\in\mathcal V_{\mathrm{n}}: X_k= \mathrm{I}} } 1+(1-\sigma_{\mathrm{v}}) \frac{1-\theta}{n-1}\sum\limits_{\substack{k\in\mathcal V_{\mathrm{v}}: X_k= \mathrm{I}} } 1\Bigg].
\end{align}
A susceptible vaccinated individual $j$ ($X_j(t^-) = \mathrm{S}$) becomes infectious ($X_j(t^+) = \mathrm{I}$) according to a Poisson clock with rate 
\begin{align}\label{kappav}
 &\kappa_{\mathrm{v},j} := 2\lambda (1-\gamma_{\mathrm{t}})\left(1-p_{\mathrm{q}}(1-\gamma_{\mathrm{q}})\right) \Bigg[(1-\sigma_{\mathrm{n}}) \frac{1-\theta}{n-1} \sum\limits_{\substack{k\in\mathcal V_{\mathrm{n}}: X_k= \mathrm{I}} } 1 +
 (1-\sigma_{\mathrm{v}})\left( \frac{\theta}{nv-1} + \frac{1-\theta}{n-1} \right)\sum\limits_{\substack{k\in\mathcal V_{\mathrm{v}}: X_k= \mathrm{I}} } 1\Bigg],
 \end{align}
while they become quarantined ($X_j(t^+) = \mathrm{Q}$) with rate 
 \begin{align}\label{nuv}
 \nu_{\mathrm{v},j} := & 2\lambda (1-\gamma_{\mathrm{t}})p_{\mathrm{q}}(1-\gamma_{\mathrm{q}})\Bigg[(1-\sigma_{\mathrm{n}}) \frac{1-\theta}{n-1}\sum\limits_{\substack{k\in\mathcal V_{\mathrm{n}}: X_k= \mathrm{I}} } 1  +
 (1-\sigma_{\mathrm{v}})\left( \frac{\theta}{nv-1} + \frac{1-\theta}{n-1} \right)\sum\limits_{\substack{k\in\mathcal V_{\mathrm{v}}: X_k= \mathrm{I}} } 1\Bigg].
\end{align}
\end{proposition}

\begin{proof}
Let us focus on the transition rates of a susceptible non-vaccinated individual $j$. Such transitions are triggered by moments of contact between $j$ and an infectious individual, which take place if one of the following events occurs: i) $j$ activates, interacts with an infected individual in $\mathcal V_{\mathrm n}$, who decides to have contact (i.e., to disregard physical distance); ii) an infected individual in $\mathcal V_{\mathrm n}$ activates and interacts with $j$, where the former decides to have contact; iii) $j$ activates, interacts with an infected individual in $\mathcal V_{\mathrm v}$, who decides to have contact; iv) an infected individual in $\mathcal V_{\mathrm v}$ activates, interacts with $j$, where the former decides to have contact. 

We observe that i) takes place if three independent events occur: first, the activation of $j$, which is triggered by a Poisson clock with unit rate; second, the selection of an infected individual $k\in\mathcal V_{\mathrm n}$, which occurs with probability equal to
\begin{equation}
 \theta\cdot \frac{1}{n(1-v)-1}\sum\limits_{\substack{k\in\mathcal V_{\mathrm{n}}: X_k= \mathrm{I}}}1 + (1-\theta)\cdot\frac{1}{n-1}\sum\limits_{\substack{k\in\mathcal V_{\mathrm{n}}: X_k= \mathrm{I}}}1,
\end{equation}
where the first term accounts for interactions within the sub-population, and the second one for random interactions; and third, the decision of $k$ of having contact, which occurs with probability equal to $1-\sigma_{\mathrm{n}}$. Item ii) of Proposition~\ref{prop:poisson_minimum} allows to compute the total rate associated with event i). Similarly, we compute the rate associated with ii), iii), and iv). Then, using item i) of Proposition~\ref{prop:poisson_minimum}, we compute 
the rate corresponding to the occurrence of contact between $j$ and an infected individual, which is equal to 
\begin{align}
 2\Bigg[\frac{\theta}{n(1-v)-1} &+ \frac{1-\theta}{n-1}\Bigg]\sum_{{k\in\mathcal V_{\mathrm{n}}: X_k= \mathrm{I}} } (1-\sigma_{\mathrm{n}}) +2 \frac{1-\theta}{n-1}\sum_{{k\in\mathcal V_{\mathrm{v}}: X_k= \mathrm{I}} } (1-\sigma_{\mathrm{v}}).
\end{align}
Finally, we observe that a non-vaccinated individual $j$ becomes infectious ($\mathrm{I}$) after contact with an infected individual with probability $\lambda \left(1-p_{\mathrm{q}}\right)$, while $j$ becomes quarantined ($\mathrm{Q}$) with probability $ \lambda p_{\mathrm{q}}$. Using again item ii) of Proposition~\ref{prop:poisson_minimum} yields \eqref{kappauv} and \eqref{nuuv}. In a similar way, we derive \eqref{kappav} and \eqref{nuv} for a susceptible vaccinated individual.
\end{proof}

For a generic individual $j\in \mathcal{V}$, 
the transition rate matrix of the Markov process $X_j(t)$ is given by
\begin{equation}
  Q_{\alpha,j} = \begin{bmatrix}
 -\kappa_{\alpha,j}-\nu_{\alpha,j} & \kappa_{\alpha,j} & \nu_{\alpha,j}\\
 \beta & -\beta-\tau& \tau \\
 \beta & 0 & -\beta
 \end{bmatrix}, \label{Q} 
\end{equation}
where $\alpha\in\{\mathrm{n},\mathrm{v}\}$ is the vaccination status of $j$, and the rows (columns) correspond to states $\mathrm{S}$, $\mathrm{I}$, and $\mathrm{Q}$, respectively.

Observe that the first row of the matrix is dependent on the states of the other population members, so it is impossible to decouple the individual dynamics. For a large-scale population, this impedes the analysis, since the dimension of the state space $X(t)$ increases exponentially with $n$. Hence, as is standard practice~\cite{virusspread,9089218}, we focus on a mean-field relaxation of the system. Specifically, rather than studying the evolution of the state of each individual $j\in\mathcal V$, we are studying the evolution of the probabilities that $j$ is in a certain state, for each $ j \in\mathcal V$. 

For a vaccination status $\alpha \in \{\mathrm{n}, \mathrm{v}\}$ and any individual $j \in \mathcal{V}$, we define the probability that $j$ is susceptible, infectious, or quarantined as 
\begin{equation}\label{eq:probabilities}s_{\alpha,j} (t) : = \mathbb{P}\left[ X_j(t) = \mathrm{S}, j\in\mathcal V_{\alpha} \right], \quad  i_{\alpha,j} (t)  : = \mathbb{P}\left[ X_j(t) = \mathrm{I}, j\in\mathcal V_{\alpha} \right], \quad \text{and} \quad  q_{\alpha,j} (t) : = \mathbb{P}\left[ X_j(t) = \mathrm{Q}, j\in\mathcal V_{\alpha} \right],\\ \end{equation}
respectively, where $s_{\mathrm{v},j} (t)=i_{\mathrm{v},j} (t)=q_{\mathrm{v},j} (t)=0$, if $j\in\mathcal V_{\mathrm{n}}$, and $s_{\mathrm{n},j} (t)=i_{\mathrm{n},j} (t)=q_{\mathrm{n},j} (t)=0$ if $j\in\mathcal V_{\mathrm{v}}$.

\section{Analysis and main results}\label{section:mainresults}

\noindent In this section, we will present the dynamics in the mean-field approximation (in Section~\ref{subsection:meanfield}), and 
do a rigorous analysis of the system (in Section~\ref{subsection:mainresults}), to shed light on the role of vaccination in the spreading of an epidemic disease. 

\subsection{Mean-field dynamics}\label{subsection:meanfield}

In the mean-field approach, following~\cite{virusspread,9089218}, we study the evolution of the expected dynamics for the probabilities in \eqref{eq:probabilities}, that is, $( \dot{s}_{\mathrm{v},j} \ \dot{i}_{\mathrm{v},j} \ \dot{q}_{\mathrm{v},j} ) = ( s_{\mathrm{v},j} \ i_{\mathrm{v},j} \ q_{\mathrm{v},j} ) \mathbb E[Q_{\mathrm{v},j}]$, and $( \dot{s}_{\mathrm{n},j} \ \dot{i}_{\mathrm{n},j} \ \dot{q}_{\mathrm{n},j} ) = ( s_{\mathrm{n},j} \ i_{\mathrm{n},j} \ q_{\mathrm{n},j} )\mathbb E[Q_{\mathrm{n},j}]$. Doing so yields the following dynamical system.

\begin{proposition}\label{prop:mean}
In the mean-field approximation, \eqref{eq:probabilities} follows
\begin{equation}\label{diffeq}
 \begin{alignedat}{2}
 & \dot{s}_{\mathrm{n},j} = && -  \lambda \bar{\alpha}_{\mathrm{n},j}s_{\mathrm{n},j} 
  + \beta i_{\mathrm{n},j} + \beta q_{\mathrm{n},j}, \\
 & \dot{i}_{\mathrm{n},j} = &&  \lambda [1-p_{\mathrm{q}}]\bar{\alpha}_{\mathrm{n},j} s_{\mathrm{n},j} - (\beta +\tau) i_{\mathrm{n},j}, \\
 &\dot{q}_{\mathrm{n},j} = &&  \lambda p_{\mathrm{q}} \bar{\alpha}_{\mathrm{n},j} s_{\mathrm{n},j} + \tau i_{\mathrm{n},j} - \beta q_{\mathrm{n},j}, \\
 & \dot{s}_{\mathrm{v},j} = && - \lambda\left(1-\gamma_{\mathrm{t}} \right) \bar{\alpha}_{\mathrm{v},j} s_{\mathrm{v},j} + \beta i_{\mathrm{v},j} + \beta q_{\mathrm{v},j}, \\
 & \dot{i}_{\mathrm{v},j} = &&  \lambda\left(1-\gamma_{\mathrm{t}} \right)\left[1-p_{\mathrm{q}}\left(1-\gamma_{\mathrm{q}} \right)\right]\bar{\alpha}_{\mathrm{v},j} s_{\mathrm{v},j}- (\beta + \tau) i_{\mathrm{v},j}, \\
 &\dot{q}_{\mathrm{v},j} = && \lambda\left(1-\gamma_{\mathrm{t}} \right)p_{\mathrm{q}}\left(1-\gamma_{\mathrm{q}} \right) \bar{\alpha}_{\mathrm{v},j} s_{\mathrm{v},j}+ \tau i_{\mathrm{v},j} - \beta q_{\mathrm{v},j}, 
\end{alignedat} 
\end{equation}
for all $j\in\mathcal V$, where
\begin{align}
 \bar{\alpha}_{\mathrm{n},j} : = & 2(1-\sigma_{\mathrm{n}})\left[ \frac{\theta}{n(1-v)-1} + \frac{1-\theta}{n-1} \right]\sum\limits_{k \in \mathcal{V} \setminus \{j\} } i_{\mathrm{n},k} +2(1-\sigma_{\mathrm{v}}) \frac{1-\theta}{n-1}\sum\limits_{k \in \mathcal{V} \setminus \{j\}} i_{\mathrm{v},k}, \label{barinteractionunv}\\
 \bar{\alpha}_{\mathrm{v},j} : = & 2(1-\sigma_{\mathrm{n}}) \frac{1-\theta}{n-1}\sum\limits_{k \in \mathcal{V} \setminus \{j\} } i_{\mathrm{n},k} + 2(1-\sigma_{\mathrm{v}})\left[ \frac{\theta}{nv-1} + \frac{1-\theta}{n-1} \right]\sum\limits_{k \in \mathcal{V} \setminus \{j\} } i_{\mathrm{v},k}.\label{barinteractionv}
\end{align}
\end{proposition}
\begin{proof}
When computing the entries $\mathbb E[\kappa_{\alpha,j}]$ and $\mathbb E[\nu_{\alpha,j}]$ of $\mathbb E[Q_{\alpha,j}]$, note that for any susceptible individual $j \in \mathcal{V}$ ($X_j(t)=\mathrm{S}$), it holds that
\begin{equation}\label{eq:mean_state}
 \mathbb E\scalemath{0.82}{\left[\sum\limits_{\substack{k\in\mathcal V_{\alpha}: X_k= \mathrm{I}} } 1\right]} =\sum\limits_{\substack{k\in\mathcal V\setminus\{j\}} } i_{\alpha,k}, 
\end{equation}
 where $\alpha\in\{\mathrm{n},\mathrm{v}\}$. The rest of the proof is obtained by substituting \eqref{eq:mean_state} into the expected dynamics.
\end{proof}

We will now show that the system in \eqref{diffeq} is well-defined by showing that $\left(s_{\mathrm{n},j} \ i_{\mathrm{n},j} \ q_{\mathrm{n},j} \ s_{\mathrm{v},j} \ i_{\mathrm{v},j} \ q_{\mathrm{v},j} \right)$ is a probability vector for all $t \in \mathbb{R}_{\ge 0}$ and all $j \in \mathcal{V}$. For this purpose, let us define the sets: $ \mathcal{S}_{\mathrm{n},j} := \{ ( s_{\mathrm{n},j} \ i_{\mathrm{n},j} \ q_{\mathrm{n},j} \ s_{\mathrm{v},j} \ i_{\mathrm{v},j} \ q_{\mathrm{v},j} ) : s_{\mathrm{n},j}, i_{\mathrm{n},j}, q_{\mathrm{n},j} \ge 0, s_{\mathrm{v},j} = i_{\mathrm{v},j} = q_{\mathrm{v},j} =0, s_{\mathrm{n},j}+ i_{\mathrm{n},j}+q_{\mathrm{n},j}=1 \}$ for any $j \in \mathcal{V}_{\mathrm{n}}$; and $ \mathcal{S}_{\mathrm{v},j} := \{ ( s_{\mathrm{n},j} \ i_{\mathrm{n},j} \ q_{\mathrm{n},j} \ s_{\mathrm{v},j} \ i_{\mathrm{v},j} \ q_{\mathrm{v},j} ) : s_{\mathrm{n},j} = i_{\mathrm{n},j} = q_{\mathrm{n},j}=0, s_{\mathrm{v},j}, i_{\mathrm{v},j}, q_{\mathrm{v},j} \ge 0, s_{\mathrm{v},j}+i_{\mathrm{v},j}+ q_{\mathrm{v},j} = 1 \}$ for any $j \in \mathcal{V}_{\mathrm{v}}$. 

\begin{lemma}\label{lemma1} For all $j\in \mathcal{V}_{\alpha}$ with $\alpha\in\{ \mathrm{n},\mathrm{v}\}$, the set
$\mathcal{S}_{\alpha,j}$ is positive invariant under \eqref{diffeq}. 
\end{lemma}
\begin{proof}
Let us consider $\mathcal{S}_{\mathrm{n},j}$ for any $j \in \mathcal{V}_{\mathrm{n}}$. First, note that if one of the probabilities governed by \eqref{diffeq} equals zero, then its respective time-derivative is non-negative. Next, observe that $s_{\mathrm{v},j} = i_{\mathrm{v},j} = q_{\mathrm{v},j} =0$ implies that the time-derivatives of $s_{\mathrm{v},j}$, $ i_{\mathrm{v},j}$, and $q_{\mathrm{v},j} =0$ are zero. Now note that for all $j \in \mathcal{V}_{\mathrm{n}}$, $ \dot{s}_{\mathrm{n},j} + \dot{i}_{\mathrm{n},j} + \dot{q}_{\mathrm{n},j} =0 $, so $s_{\mathrm{n},j}+ i_{\mathrm{n},j}+q_{\mathrm{n},j} = 1 $ for all $t \in \mathbb{R}_{\ge 0}$. For any $j \in \mathcal{V}_{\mathrm{v}}$, the proof works analogously for $\mathcal{S}_{\mathrm{v},j}$ . 
\end{proof}

To commence the mean-field analysis of the system, let us denote the average probability for a randomly selected individual to have vaccination status $\alpha$ and be susceptible, infectious, or quarantined as
\begin{equation}\label{propapprox}
 y_{\alpha,\mathrm{s}} := \dfrac{1}{n} \sum_{j \in \mathcal{V}} s_{\alpha,j}, \quad y_{\alpha,\mathrm{i}}:= \dfrac{1}{n} \sum_{j \in \mathcal{V}} i_{\alpha,j}, \quad \text{and} \quad y_{\alpha,\mathrm{q}} := \dfrac{1}{n} \sum_{j \in \mathcal{V}} q_{\alpha,j},
\end{equation}
respectively, where $\alpha\in\{\mathrm{n},\mathrm{v}\}$.

By taking a sufficiently large population size $n$, the fraction of individuals in a state can be arbitrarily closely approximated by the average probability to be in that state, for any finite time-horizon~\cite{limittheo,zino2}, i.e.,
\begin{equation}\label{eq:macro_variables}
 \begin{alignedat}{2}
  S_{\alpha}(t) &:= \tfrac{1}{n}| \{j \in \mathcal{V}_{\alpha} : X_j(t) = \mathrm{S} \} | &&\approx y_{\alpha,\mathrm{s}}, \\
  I_{\alpha}(t) &:= \tfrac{1}{n}| \{j \in \mathcal{V}_{\alpha} : X_j(t) = \mathrm{I} \} | &&\approx y_{\alpha,\mathrm{i}}, \\
  Q_{\alpha}(t) &:= \tfrac{1}{n}| \{j \in \mathcal{V}_{\alpha} : X_j(t) = \mathrm{Q} \} |&& \approx y_{\alpha,\mathrm{q}}, \\
\end{alignedat} 
\end{equation}
with $\alpha\in\{\mathrm{n},\mathrm{v}\}$, as illustrated in Fig.~\ref{fig:approximation}.

\begin{figure}
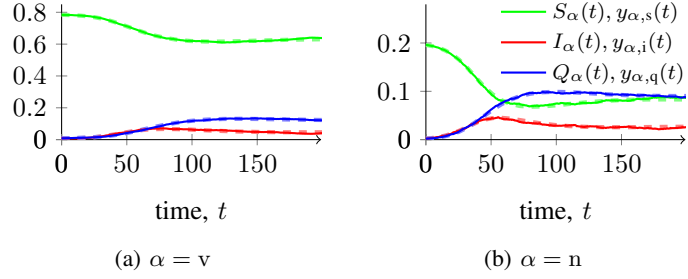

\vspace{0.25cm}
\centering
\subfloat[$\alpha=\mathrm{v}$]{\input{Figures/2a}} \hspace{0.5cm}
\subfloat[$\alpha=\mathrm{n}$]{\input{Figures/2b}}
 \caption{Comparison of the quantities in \eqref{eq:macro_variables} between a simulation of the Markov process (solid) and its deterministic approximation from Proposition~\ref{prop:dynamics} (dashed). Parameters are $n=20\,000$, $v=0.8$, $\lambda=0.2$, $\sigma_{\mathrm{v}}=0.7$, $\sigma_{\mathrm{n}}=0.2$, $p_{\mathrm{q}}=0.2$, $\beta=0.02$, $\gamma_{\mathrm{t}}=0.5$, $\gamma_{\mathrm{q}}=0.9$, $\tau=0.05$, and $\theta=0.5$. } \label{fig:approximation} 
\end{figure}



Since the average probabilities in \eqref{propapprox} adequately reflect the state of a sufficiently large population, we now focus on the dynamics of the macroscopic variables in \eqref{propapprox}, presented in the following proposition. 
\begin{proposition}\label{prop:dynamics}
Consider the system in \eqref{diffeq}. In the thermodynamic limit of large-scale systems $n\to \infty$, the dynamics of \eqref{propapprox} are given by
\begin{align}\label{y}
 \dot{y}_{\mathrm{n},\mathrm{s}} = & - 2  \lambda \Big( \tfrac{\theta}{1-v} + 1 -\theta\Big)(1-\sigma_{\mathrm{n}})y_{\mathrm{n},\mathrm{s}} y_{\mathrm{n},\mathrm{i}}  - 2  \lambda (1-\theta)(1-\sigma_{\mathrm{v}}) y_{\mathrm{n},\mathrm{s}} y_{\mathrm{v},\mathrm{i}} \notag+ \beta y_{\mathrm{n},\mathrm{i}} + \beta y_{\mathrm{n},\mathrm{q}}, \notag\\
 \dot{y}_{\mathrm{n},\mathrm{i}} \hspace{1pt} = & 2  \lambda [1-p_{\mathrm{q}}]  \Big( \tfrac{\theta}{1-v} + 1 -\theta\Big)(1-\sigma_{\mathrm{n}})y_{\mathrm{n},\mathrm{s}} y_{\mathrm{n},\mathrm{i}} + 2 \lambda [1-p_{\mathrm{q}}] (1-\theta)(1-\sigma_{\mathrm{v}}) y_{\mathrm{n},\mathrm{s}} y_{\mathrm{v},\mathrm{i}} \notag - (\beta + \tau) y_{\mathrm{n},\mathrm{i}}, \notag\\
 \dot{y}_{\mathrm{n},\mathrm{q}} = & 2  \lambda p_{\mathrm{q}}  \Big( \tfrac{\theta}{1-v} + 1 -\theta\Big)(1-\sigma_{\mathrm{n}})y_{\mathrm{n},\mathrm{s}} y_{\mathrm{n},\mathrm{i}}  +  2\lambda p_{\mathrm{q}} (1-\theta)(1-\sigma_{\mathrm{v}}) y_{\mathrm{n},\mathrm{s}} y_{\mathrm{v},\mathrm{i}}  + \tau y_{\mathrm{n},\mathrm{i}} - \beta y_{\mathrm{n},\mathrm{q}}, \notag\\
 \dot{y}_{\mathrm{v},\mathrm{s}} =& - 2  \lambda (1-\gamma_{\mathrm{t}}) (1-\theta)(1-\sigma_{\mathrm{n}}) y_{\mathrm{v},\mathrm{s}} y_{\mathrm{n},\mathrm{i}} + \beta y_{\mathrm{v},\mathrm{i}} - 2  \lambda (1-\gamma_{\mathrm{t}}) \Big( \tfrac{\theta}{v} + 1 -\theta\Big)(1-\sigma_{\mathrm{v}})y_{\mathrm{v},\mathrm{s}} y_{\mathrm{v},\mathrm{i}} + \beta y_{\mathrm{v},\mathrm{q}}, \\
 \dot{y}_{\mathrm{v},\mathrm{i}} \hspace{1pt} = & 2 \lambda (1-\gamma_{\mathrm{t}}) [1-p_{\mathrm{q}}(1-\gamma_{\mathrm{q}})] (1-\theta)(1-\sigma_{\mathrm{n}}) y_{\mathrm{v},\mathrm{s}} y_{\mathrm{n},\mathrm{i}}  + 2  \lambda (1-\gamma_{\mathrm{t}}) [1-p_{\mathrm{q}}(1-\gamma_{\mathrm{q}})]  \Big( \tfrac{\theta}{v} + 1 -\theta\Big)(1-\sigma_{\mathrm{v}})y_{\mathrm{v},\mathrm{s}} y_{\mathrm{v},\mathrm{i}} \notag \\
 & - (\beta + \tau) y_{\mathrm{v},\mathrm{i}}, \notag \\
 \dot{y}_{\mathrm{v},\mathrm{q}} \hspace{1pt}= & 2 \lambda (1-\gamma_{\mathrm{t}}) p_{\mathrm{q}}(1-\gamma_{\mathrm{q}})(1-\theta)(1-\sigma_{\mathrm{n}}) y_{\mathrm{v},\mathrm{s}} y_{\mathrm{n},\mathrm{i}} + 2  \lambda (1-\gamma_{\mathrm{t}}) p_{\mathrm{q}}(1-\gamma_{\mathrm{q}}) \Big( \tfrac{\theta}{v} + 1 -\theta\Big)(1-\sigma_{\mathrm{v}})y_{\mathrm{v},\mathrm{s}} y_{\mathrm{v},\mathrm{i}}  + \tau y_{\mathrm{v},\mathrm{i}} - \beta y_{\mathrm{v},\mathrm{q}}. \notag
\end{align}
\end{proposition}
\begin{proof}
The equations are obtained by computing the temporal derivatives of the expressions in \eqref{propapprox}, and by using the expressions derived in Proposition~\ref{prop:mean}.
\end{proof}

\begin{remark}\label{rem:4}
Only $4$ of the $6$ equations in \eqref{y} are linearly independent, since ${y}_{\mathrm{v},\mathrm{s}}+{y}_{\mathrm{v},\mathrm{i}}+{y}_{\mathrm{v},\mathrm{q}}=v$ and ${y}_{\mathrm{n},\mathrm{s}}+{y}_{\mathrm{n},\mathrm{i}}+{y}_{\mathrm{n},\mathrm{q}}=1-v$. 
\end{remark}

\subsection{Epidemic threshold} \label{subsection:mainresults}

Here, we study whether a local outbreak of the infection leads to endemicity in the population. 
Theorem \ref{theo1} presents the \textit{epidemic threshold}, that is, the conditions such that the disease-free equilibrium (DFE) of \eqref{y} (i.e., $y_{ \mathrm{n},\mathrm{i}}=y_{ \mathrm{n},\mathrm{q}}=y_{ \mathrm{v},\mathrm{i}}=y_{ \mathrm{v},\mathrm{q}}=0$) is (locally) asymptotically stable. The threshold is formulated as a critical value for the control input, that is, the testing rate $\tau$; if the rate of testing $\tau$ is higher than the threshold $\bar{\tau}$, then the local outbreak will be eradicated. If not, it will lead to endemicity. The proof is reported in Appendix~\ref{prooftheo1}. 

\begin{theorem} \label{theo1} Consider the system in \eqref{y}. In the thermodynamic limit of large-scale systems $n\to \infty$, the epidemic threshold is equal to
\begin{equation}\label{eq:threshold}
 \bar{\tau} := \lambda \xi - \beta+ \lambda \sqrt{ \xi^2 - 4 \theta (1-\gamma_{\mathrm{t}})[1-p_{\mathrm{q}}][1-p_{\mathrm{q}}(1-\gamma_{\mathrm{q}})] (1-\sigma_{\mathrm{n}})(1-\sigma_{\mathrm{v}})}, 
\end{equation}
where
\begin{equation}\label{xi}
 \xi :=  (1-p_{\mathrm{q}})(1-\sigma_{\mathrm{n}})[\theta + (1-\theta)(1-v)] + (1-\gamma_{\mathrm{t}})[1-p_{\mathrm{q}}(1-\gamma_{\mathrm{q}})](1-\sigma_{\mathrm{v}})[ \theta +(1-\theta)v].
\end{equation}
If $ \tau > \bar{\tau}$, the DFE is locally asymptotically stable.
\end{theorem}

\begin{remark}
From \eqref{eq:threshold}, observe that if the recovery rate satisfies
\begin{equation}\label{remark5beta}
  \beta>\lambda \xi+ \lambda \sqrt{ \xi^2 - 4 \theta (1-\gamma_{\mathrm{t}})[1-p_{\mathrm{q}}][1-p_{\mathrm{q}}(1-\gamma_{\mathrm{q}})] (1-\sigma_{\mathrm{n}})(1-\sigma_{\mathrm{v}}) }, 
\end{equation}
then no control is needed, since the DFE is always locally asymptotically stable.
\end{remark}

\begin{corollary}
In the absence of homophily ($\theta=0$), that is, when individuals' interactions are not influenced by the multi-population structure, the epidemic threshold in \eqref{eq:threshold} reduces to 
\begin{equation}\label{threshold2step}
  \tau ^{\ast} := 2\lambda (1-p_{\mathrm{q}})(1-v)(1-\sigma_{\mathrm{n}}) + 2\lambda(1-\gamma_{\mathrm{t}})v[1-p_{\mathrm{q}}(1-\gamma_{\mathrm{q}})](1-\sigma_{\mathrm{v}}) - \beta. 
\end{equation}
\end{corollary}

\begin{figure}
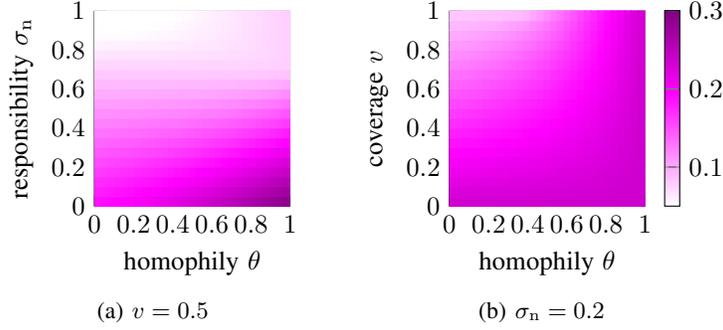

\vspace{0.25cm}
\centering
 \subfloat[$v=0.5$]{\input{Figures/3a}\label{fig:3a}} \hspace{0.5cm}
 \subfloat[$\sigma_{\mathrm{n}}=0.2$]{ \input{Figures/3b}\label{fig:3b}}
 \caption{The epidemic threshold (color-coded), computed using \eqref{eq:threshold} for different values of the model parameters. Common parameters are $\lambda=0.2$, $\sigma_{\mathrm{v}}=0.5$, $p_{\mathrm{q}}=0.2$, $\beta=0.02$, $\gamma_{\mathrm{t}}=0.5$, $\gamma_{\mathrm{q}}=0.9$, and $\tau=0.05$. }
 \label{fig:threshold}
\end{figure}

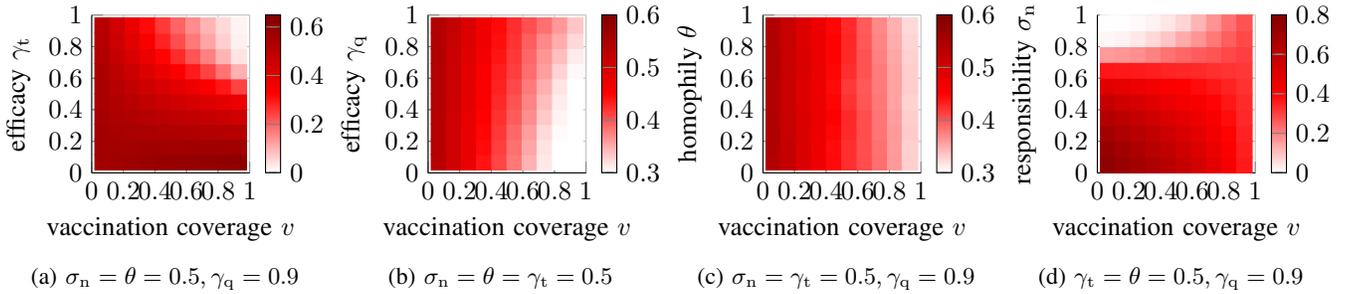
\begin{figure*}
\centering
\subfloat[$\sigma_{\mathrm{n}}=\theta=0.5,\gamma_{\mathrm{q}}=0.9$]{\begin{tikzpicture}
\begin{axis}[%
colormap={mymap}{[1pt] rgb(0pt)=(1,1,1); rgb(127pt)=(1,0,0); rgb(255pt)=(.5,0,0)},
width=\q cm,
height=\q cm,
scale only axis,
xmin=0,
xmax=1,
ymin=0,
point meta max=.65,
point meta min=0,
ymax=1,
axis background/.style={fill=white},
axis x line*=bottom,
axis y line*=left,
xlabel={vaccination coverage $v$},
ylabel={efficacy $\gamma_{\mathrm{t}}$},
xmajorgrids,colorbar,colorbar style={  at={(1.1,1)},width=.2cm, yticklabel style={
        /pgf/number format/fixed,
        /pgf/number format/precision=5
},},
ymajorgrids,
]

\addplot[%
surf,
shader=flat, draw=none, mesh/rows=11]
table[row sep=crcr, point meta=\thisrow{c}] {%
x	y	c\\
0.02	0.02	0.55076\\
0.02	0.116	0.55146\\
0.02	0.212	0.55466\\
0.02	0.308	0.55962\\
0.02	0.404	0.55282\\
0.02	0.5	0.55466\\
0.02	0.596	0.55053\\
0.02	0.692	0.55161\\
0.02	0.788	0.54651\\
0.02	0.884	0.54633\\
0.02	0.98	0.53646\\
0.116	0.02	0.5745\\
0.116	0.116	0.55915\\
0.116	0.212	0.55519\\
0.116	0.308	0.54298\\
0.116	0.404	0.52944\\
0.116	0.5	0.5275\\
0.116	0.596	0.52065\\
0.116	0.692	0.50613\\
0.116	0.788	0.48653\\
0.116	0.884	0.47781\\
0.116	0.98	0.46642\\
0.212	0.02	0.57327\\
0.212	0.116	0.56782\\
0.212	0.212	0.55187\\
0.212	0.308	0.53552\\
0.212	0.404	0.52677\\
0.212	0.5	0.50429\\
0.212	0.596	0.48362\\
0.212	0.692	0.46009\\
0.212	0.788	0.42934\\
0.212	0.884	0.41352\\
0.212	0.98	0.397\\
0.308	0.02	0.58276\\
0.308	0.116	0.56472\\
0.308	0.212	0.55323\\
0.308	0.308	0.53227\\
0.308	0.404	0.50203\\
0.308	0.5	0.47983\\
0.308	0.596	0.4452\\
0.308	0.692	0.41709\\
0.308	0.788	0.38618\\
0.308	0.884	0.35608\\
0.308	0.98	0.33195\\
0.404	0.02	0.59336\\
0.404	0.116	0.58046\\
0.404	0.212	0.55353\\
0.404	0.308	0.52482\\
0.404	0.404	0.49068\\
0.404	0.5	0.45556\\
0.404	0.596	0.41541\\
0.404	0.692	0.36653\\
0.404	0.788	0.32362\\
0.404	0.884	0.29603\\
0.404	0.98	0.26746\\
0.5	0.02	0.6\\
0.5	0.116	0.57088\\
0.5	0.212	0.54193\\
0.5	0.308	0.5176\\
0.5	0.404	0.47256\\
0.5	0.5	0.43589\\
0.5	0.596	0.38434\\
0.5	0.692	0.32057\\
0.5	0.788	0.27431\\
0.5	0.884	0.23056\\
0.5	0.98	0.21422\\
0.596	0.02	0.60569\\
0.596	0.116	0.57993\\
0.596	0.212	0.54197\\
0.596	0.308	0.50895\\
0.596	0.404	0.46113\\
0.596	0.5	0.41017\\
0.596	0.596	0.34968\\
0.596	0.692	0.27447\\
0.596	0.788	0.21286\\
0.596	0.884	0.18253\\
0.596	0.98	0.15714\\
0.692	0.02	0.6132\\
0.692	0.116	0.58334\\
0.692	0.212	0.54848\\
0.692	0.308	0.4966\\
0.692	0.404	0.45775\\
0.692	0.5	0.39008\\
0.692	0.596	0.30609\\
0.692	0.692	0.22609\\
0.692	0.788	0.15462\\
0.692	0.884	0.12845\\
0.692	0.98	0.1101\\
0.788	0.02	0.62046\\
0.788	0.116	0.5837\\
0.788	0.212	0.53841\\
0.788	0.308	0.48879\\
0.788	0.404	0.43805\\
0.788	0.5	0.36334\\
0.788	0.596	0.26838\\
0.788	0.692	0.16084\\
0.788	0.788	0.10007\\
0.788	0.884	0.07321\\
0.788	0.98	0.06402\\
0.884	0.02	0.62933\\
0.884	0.116	0.5894\\
0.884	0.212	0.54254\\
0.884	0.308	0.48805\\
0.884	0.404	0.4211\\
0.884	0.5	0.34274\\
0.884	0.596	0.20216\\
0.884	0.692	0.08052\\
0.884	0.788	0.04076\\
0.884	0.884	0.03204\\
0.884	0.98	0.02765\\
0.98	0.02	0.63897\\
0.98	0.116	0.59217\\
0.98	0.212	0.53529\\
0.98	0.308	0.47453\\
0.98	0.404	0.41783\\
0.98	0.5	0.31201\\
0.98	0.596	0.10184\\
0.98	0.692	0.01529\\
0.98	0.788	0.0034\\
0.98	0.884	0.00228\\
0.98	0.98	0.00191\\
};

\end{axis}

\end{tikzpicture}%
}
\subfloat[$\sigma_{\mathrm{n}}=\theta=\gamma_{\mathrm{t}}=0.5$]{\begin{tikzpicture}
\begin{axis}[%
colormap={mymap}{[1pt] rgb(0pt)=(1,1,1); rgb(127pt)=(1,0,0); rgb(255pt)=(.5,0,0)},
width=\q cm,
height=\q cm,
scale only axis,
xmin=0,
xmax=1,
ymin=0,
point meta max=.6,
point meta min=.3,
ymax=1,
axis background/.style={fill=white},
axis x line*=bottom,
axis y line*=left,
ylabel={efficacy $\gamma_{\mathrm{q}}$},
xlabel={vaccination coverage $v$},
colorbar,
colorbar style={  at={(1.1,1)},width=.2cm, yticklabel style={
        /pgf/number format/fixed,
        /pgf/number format/precision=5
},},
xmajorgrids,
ymajorgrids,
]

\addplot[%
surf,
shader=flat, draw=none, mesh/rows=11]
table[row sep=crcr, point meta=\thisrow{c}] {%
x	y	c\\
0.02	0.02	0.54467\\
0.02	0.116	0.5529\\
0.02	0.212	0.55264\\
0.02	0.308	0.55016\\
0.02	0.404	0.55047\\
0.02	0.5	0.54949\\
0.02	0.596	0.55381\\
0.02	0.692	0.5528\\
0.02	0.788	0.55347\\
0.02	0.884	0.54865\\
0.02	0.98	0.55413\\
0.116	0.02	0.51794\\
0.116	0.116	0.51707\\
0.116	0.212	0.52195\\
0.116	0.308	0.52056\\
0.116	0.404	0.52525\\
0.116	0.5	0.52642\\
0.116	0.596	0.52296\\
0.116	0.692	0.52442\\
0.116	0.788	0.53034\\
0.116	0.884	0.52563\\
0.116	0.98	0.52833\\
0.212	0.02	0.48913\\
0.212	0.116	0.49422\\
0.212	0.212	0.49144\\
0.212	0.308	0.49204\\
0.212	0.404	0.49217\\
0.212	0.5	0.49614\\
0.212	0.596	0.4988\\
0.212	0.692	0.49759\\
0.212	0.788	0.4977\\
0.212	0.884	0.50757\\
0.212	0.98	0.51217\\
0.308	0.02	0.45053\\
0.308	0.116	0.45679\\
0.308	0.212	0.4614\\
0.308	0.308	0.45905\\
0.308	0.404	0.46413\\
0.308	0.5	0.46987\\
0.308	0.596	0.46704\\
0.308	0.692	0.47357\\
0.308	0.788	0.47825\\
0.308	0.884	0.47557\\
0.308	0.98	0.47747\\
0.404	0.02	0.42045\\
0.404	0.116	0.42645\\
0.404	0.212	0.43021\\
0.404	0.308	0.43428\\
0.404	0.404	0.43715\\
0.404	0.5	0.43949\\
0.404	0.596	0.44016\\
0.404	0.692	0.44206\\
0.404	0.788	0.45551\\
0.404	0.884	0.45157\\
0.404	0.98	0.46032\\
0.5	0.02	0.39326\\
0.5	0.116	0.39501\\
0.5	0.212	0.40182\\
0.5	0.308	0.40153\\
0.5	0.404	0.41268\\
0.5	0.5	0.41534\\
0.5	0.596	0.42711\\
0.5	0.692	0.42169\\
0.5	0.788	0.43036\\
0.5	0.884	0.42802\\
0.5	0.98	0.43612\\
0.596	0.02	0.36028\\
0.596	0.116	0.36868\\
0.596	0.212	0.37074\\
0.596	0.308	0.37917\\
0.596	0.404	0.39046\\
0.596	0.5	0.38759\\
0.596	0.596	0.3976\\
0.596	0.692	0.4066\\
0.596	0.788	0.40679\\
0.596	0.884	0.41402\\
0.596	0.98	0.41075\\
0.692	0.02	0.3228\\
0.692	0.116	0.33121\\
0.692	0.212	0.33431\\
0.692	0.308	0.3459\\
0.692	0.404	0.35135\\
0.692	0.5	0.35571\\
0.692	0.596	0.36827\\
0.692	0.692	0.3704\\
0.692	0.788	0.37169\\
0.692	0.884	0.3855\\
0.692	0.98	0.40242\\
0.788	0.02	0.28304\\
0.788	0.116	0.28979\\
0.788	0.212	0.29641\\
0.788	0.308	0.30965\\
0.788	0.404	0.31899\\
0.788	0.5	0.32709\\
0.788	0.596	0.33631\\
0.788	0.692	0.35102\\
0.788	0.788	0.35469\\
0.788	0.884	0.36702\\
0.788	0.98	0.37872\\
0.884	0.02	0.21817\\
0.884	0.116	0.24301\\
0.884	0.212	0.24742\\
0.884	0.308	0.25413\\
0.884	0.404	0.27989\\
0.884	0.5	0.29981\\
0.884	0.596	0.29774\\
0.884	0.692	0.31688\\
0.884	0.788	0.32799\\
0.884	0.884	0.34454\\
0.884	0.98	0.34974\\
0.98	0.02	0.13307\\
0.98	0.116	0.15235\\
0.98	0.212	0.18297\\
0.98	0.308	0.19351\\
0.98	0.404	0.2152\\
0.98	0.5	0.22431\\
0.98	0.596	0.25993\\
0.98	0.692	0.27503\\
0.98	0.788	0.30293\\
0.98	0.884	0.31464\\
0.98	0.98	0.32675\\
};

\end{axis}

\end{tikzpicture}%
}
\subfloat[$\sigma_{\mathrm{n}}=\gamma_{\mathrm{t}}=0.5,\gamma_{\mathrm{q}}=0.9$]{\begin{tikzpicture}
\begin{axis}[%
colormap={mymap}{[1pt] rgb(0pt)=(1,1,1); rgb(127pt)=(1,0,0); rgb(255pt)=(.5,0,0)},
width=\q cm,
height=\q cm,
scale only axis,
xmin=0,
xmax=1,
ymin=0,
point meta max=.6,
point meta min=0.3,
ymax=1,
axis background/.style={fill=white},
axis x line*=bottom,
axis y line*=left,
xlabel={vaccination coverage $v$},
ylabel={homophily $\theta$},
xmajorgrids,colorbar,colorbar style={  at={(1.1,1)},width=.2cm, yticklabel style={
        /pgf/number format/fixed,
        /pgf/number format/precision=5
},},
ymajorgrids,
]

\addplot[%
surf,
shader=flat, draw=none, mesh/rows=11]
table[row sep=crcr, point meta=\thisrow{c}] {%
x	y	c\\
0.02	0.02	0.54583\\
0.02	0.116	0.55181\\
0.02	0.212	0.55325\\
0.02	0.308	0.55602\\
0.02	0.404	0.54954\\
0.02	0.5	0.54652\\
0.02	0.596	0.55077\\
0.02	0.692	0.55268\\
0.02	0.788	0.54749\\
0.02	0.884	0.54871\\
0.02	0.98	0.55288\\
0.116	0.02	0.52318\\
0.116	0.116	0.53195\\
0.116	0.212	0.52669\\
0.116	0.308	0.52862\\
0.116	0.404	0.52561\\
0.116	0.5	0.52883\\
0.116	0.596	0.52644\\
0.116	0.692	0.52477\\
0.116	0.788	0.53041\\
0.116	0.884	0.53286\\
0.116	0.98	0.52686\\
0.212	0.02	0.50271\\
0.212	0.116	0.51092\\
0.212	0.212	0.50111\\
0.212	0.308	0.50605\\
0.212	0.404	0.50157\\
0.212	0.5	0.49952\\
0.212	0.596	0.50056\\
0.212	0.692	0.50706\\
0.212	0.788	0.49728\\
0.212	0.884	0.50137\\
0.212	0.98	0.50206\\
0.308	0.02	0.48093\\
0.308	0.116	0.47925\\
0.308	0.212	0.48099\\
0.308	0.308	0.48053\\
0.308	0.404	0.48224\\
0.308	0.5	0.47589\\
0.308	0.596	0.48279\\
0.308	0.692	0.48377\\
0.308	0.788	0.47856\\
0.308	0.884	0.4851\\
0.308	0.98	0.47702\\
0.404	0.02	0.46076\\
0.404	0.116	0.45641\\
0.404	0.212	0.45901\\
0.404	0.308	0.46052\\
0.404	0.404	0.4606\\
0.404	0.5	0.46471\\
0.404	0.596	0.45834\\
0.404	0.692	0.46234\\
0.404	0.788	0.45896\\
0.404	0.884	0.45615\\
0.404	0.98	0.45497\\
0.5	0.02	0.42953\\
0.5	0.116	0.42681\\
0.5	0.212	0.42897\\
0.5	0.308	0.43433\\
0.5	0.404	0.44142\\
0.5	0.5	0.4217\\
0.5	0.596	0.43118\\
0.5	0.692	0.43451\\
0.5	0.788	0.43601\\
0.5	0.884	0.44255\\
0.5	0.98	0.42914\\
0.596	0.02	0.41912\\
0.596	0.116	0.41163\\
0.596	0.212	0.42013\\
0.596	0.308	0.41738\\
0.596	0.404	0.40971\\
0.596	0.5	0.40793\\
0.596	0.596	0.40581\\
0.596	0.692	0.41794\\
0.596	0.788	0.41456\\
0.596	0.884	0.41357\\
0.596	0.98	0.41122\\
0.692	0.02	0.39429\\
0.692	0.116	0.39443\\
0.692	0.212	0.38862\\
0.692	0.308	0.39121\\
0.692	0.404	0.39671\\
0.692	0.5	0.38381\\
0.692	0.596	0.39437\\
0.692	0.692	0.38531\\
0.692	0.788	0.38534\\
0.692	0.884	0.38178\\
0.692	0.98	0.37813\\
0.788	0.02	0.37494\\
0.788	0.116	0.37347\\
0.788	0.212	0.37532\\
0.788	0.308	0.36932\\
0.788	0.404	0.36964\\
0.788	0.5	0.36783\\
0.788	0.596	0.36621\\
0.788	0.692	0.3656\\
0.788	0.788	0.36019\\
0.788	0.884	0.36081\\
0.788	0.98	0.35448\\
0.884	0.02	0.3457\\
0.884	0.116	0.34233\\
0.884	0.212	0.34335\\
0.884	0.308	0.34898\\
0.884	0.404	0.33891\\
0.884	0.5	0.34001\\
0.884	0.596	0.34323\\
0.884	0.692	0.3411\\
0.884	0.788	0.33948\\
0.884	0.884	0.33416\\
0.884	0.98	0.33654\\
0.98	0.02	0.31463\\
0.98	0.116	0.30987\\
0.98	0.212	0.31497\\
0.98	0.308	0.31709\\
0.98	0.404	0.31308\\
0.98	0.5	0.31432\\
0.98	0.596	0.31952\\
0.98	0.692	0.3136\\
0.98	0.788	0.31559\\
0.98	0.884	0.31114\\
0.98	0.98	0.30722\\
};

\end{axis}

\end{tikzpicture}%
}
\subfloat[$\gamma_{\mathrm{t}}=\theta=0.5,\gamma_{\mathrm{q}}=0.9$]{\begin{tikzpicture}
\begin{axis}[%
colormap={mymap}{[1pt] rgb(0pt)=(1,1,1); rgb(127pt)=(1,0,0); rgb(255pt)=(.5,0,0)},
width=\q cm,
height=\q cm,
scale only axis,
xmin=0,
xmax=1,
ymin=0,
point meta max=.8,
point meta min=0,
ymax=1,
axis background/.style={fill=white},
axis x line*=bottom,
axis y line*=left,
xlabel={vaccination coverage $v$},
ylabel={responsibility $\sigma_{\mathrm{n}}$},
colorbar,
colorbar style={  at={(1.1,1)},width=.2cm, yticklabel style={
        /pgf/number format/fixed,
        /pgf/number format/precision=5
},},
xmajorgrids,
ymajorgrids,
]

\addplot[%
surf,
shader=flat, draw=none, mesh/rows=11]
table[row sep=crcr, point meta=\thisrow{c}] {%
x	y	c\\
0.02	0	0.77696\\
0.02	0.1	0.75139\\
0.02	0.2	0.72186\\
0.02	0.3	0.67994\\
0.02	0.4	0.62985\\
0.02	0.5	0.55627\\
0.02	0.6	0.4488\\
0.02	0.7	0.28861\\
0.02	0.8	0.01859\\
0.02	0.9	0.00092\\
0.02	1	0.00028\\
0.116	0	0.7473\\
0.116	0.1	0.7212\\
0.116	0.2	0.69498\\
0.116	0.3	0.65354\\
0.116	0.4	0.59866\\
0.116	0.5	0.52844\\
0.116	0.6	0.43437\\
0.116	0.7	0.29151\\
0.116	0.8	0.02938\\
0.116	0.9	0.00239\\
0.116	1	0.00117\\
0.212	0	0.71876\\
0.212	0.1	0.68757\\
0.212	0.2	0.6596\\
0.212	0.3	0.62076\\
0.212	0.4	0.57252\\
0.212	0.5	0.50541\\
0.212	0.6	0.42493\\
0.212	0.7	0.29215\\
0.212	0.8	0.04612\\
0.212	0.9	0.00698\\
0.212	1	0.0031\\
0.308	0	0.68206\\
0.308	0.1	0.65995\\
0.308	0.2	0.62682\\
0.308	0.3	0.5893\\
0.308	0.4	0.54692\\
0.308	0.5	0.47753\\
0.308	0.6	0.41308\\
0.308	0.7	0.29897\\
0.308	0.8	0.06731\\
0.308	0.9	0.02275\\
0.308	1	0.00719\\
0.404	0	0.65305\\
0.404	0.1	0.62447\\
0.404	0.2	0.59333\\
0.404	0.3	0.55698\\
0.404	0.4	0.51119\\
0.404	0.5	0.45412\\
0.404	0.6	0.3985\\
0.404	0.7	0.30269\\
0.404	0.8	0.11309\\
0.404	0.9	0.03379\\
0.404	1	0.02004\\
0.5	0	0.60871\\
0.5	0.1	0.5889\\
0.5	0.2	0.56386\\
0.5	0.3	0.52568\\
0.5	0.4	0.47332\\
0.5	0.5	0.42241\\
0.5	0.6	0.3811\\
0.5	0.7	0.29475\\
0.5	0.8	0.16239\\
0.5	0.9	0.061\\
0.5	1	0.0393\\
0.596	0	0.57148\\
0.596	0.1	0.54551\\
0.596	0.2	0.5179\\
0.596	0.3	0.48188\\
0.596	0.4	0.4507\\
0.596	0.5	0.41586\\
0.596	0.6	0.37654\\
0.596	0.7	0.30291\\
0.596	0.8	0.19106\\
0.596	0.9	0.10861\\
0.596	1	0.06085\\
0.692	0	0.51665\\
0.692	0.1	0.4954\\
0.692	0.2	0.47258\\
0.692	0.3	0.45197\\
0.692	0.4	0.42494\\
0.692	0.5	0.38429\\
0.692	0.6	0.35024\\
0.692	0.7	0.30655\\
0.692	0.8	0.20952\\
0.692	0.9	0.15999\\
0.692	1	0.13855\\
0.788	0	0.46667\\
0.788	0.1	0.4487\\
0.788	0.2	0.42714\\
0.788	0.3	0.40734\\
0.788	0.4	0.3868\\
0.788	0.5	0.37186\\
0.788	0.6	0.3486\\
0.788	0.7	0.29328\\
0.788	0.8	0.25926\\
0.788	0.9	0.21121\\
0.788	1	0.17095\\
0.884	0	0.40473\\
0.884	0.1	0.3893\\
0.884	0.2	0.37519\\
0.884	0.3	0.36873\\
0.884	0.4	0.35041\\
0.884	0.5	0.33838\\
0.884	0.6	0.32159\\
0.884	0.7	0.31231\\
0.884	0.8	0.28043\\
0.884	0.9	0.25637\\
0.884	1	0.24707\\
0.98	0	0.32285\\
0.98	0.1	0.32304\\
0.98	0.2	0.32108\\
0.98	0.3	0.31623\\
0.98	0.4	0.31566\\
0.98	0.5	0.3176\\
0.98	0.6	0.30943\\
0.98	0.7	0.3071\\
0.98	0.8	0.30314\\
0.98	0.9	0.29536\\
0.98	1	0.28502\\
};

\end{axis}

\end{tikzpicture}%
\label{fig:endemic_responsability}}
\caption{Total fraction of infected individuals ($I_{\mathrm{v}}(t)+I_{\mathrm{n}}(t)+Q_{\mathrm{v}}(t)+Q_{\mathrm{n}}(t)$) at $T=200$ for different levels of the vaccination coverage $v$, with on the $y$-axis varying levels of: (a) the vaccine effectiveness against transmission $\gamma_{\mathrm{t}}$ and (b) severe symptoms $\gamma_{\mathrm{q}}$; (c) the level of homophily $\theta$; and the responsibility of non-vaccinated individuals $\sigma_{\mathrm{n}}$. Common parameters are $n=10\,000$, $\lambda=0.2$, $\beta=0.02$, $\sigma_{\mathrm{v}}=0.5$, $p_{\mathrm{q}}=0.2$, and $\tau=0.05$. Each data-point is obtained by averaging $10$ independent runs of the Markov process. }
 \label{fig:endemic}
\end{figure*}

\noindent In general, the expression of the epidemic threshold in Theorem~\ref{theo1} is non-trivial. On the one hand, one can immediately observe the following monotonicity properties, which can be derived directly from the stability analysis of the DFE, as reported in Appendix~\ref{app:monotonicity}. Doing so leads to the following proposition. 
\begin{proposition}\label{prop:monotonicity}
An increase in the infection probability $\lambda$ or in the effectiveness of the vaccine against severe illness $\gamma_{\mathrm{q}}$ increases the epidemic threshold $\bar{\tau}$ in \eqref{eq:threshold}; on the other hand, $\bar{\tau}$ decreases for an increase in the recovery rate $\beta$, the effectiveness of the vaccine against transmission $\gamma_{\mathrm{t}}$, the responsibilities $\sigma_{\mathrm{v}}$ and $\sigma_{\mathrm{n}}$, and the probability of developing severe illness $p_{\mathrm{q}}$.
\end{proposition}

Interestingly, Proposition~\ref{prop:monotonicity} states that the vaccine effectiveness against transmission always facilitates the eradication of an epidemic outbreak, while its effectiveness against severe illness increases the threshold, hindering the controllability of the disease. This is due to the fact that infected individuals with non-severe symptoms may not be detected, increasing the disease transmission, which may offer an explanation for the epidemic outbreaks that occurred after the COVID-19 vaccination campaign. 

On the other hand, the impact of the remaining parameters (the vaccination coverage $v$ and the homophily $\theta$) is less intuitive. Of particular interest is to understand the role of vaccination coverage in the facilitation or deterrence of epidemic outbreaks. To shed light on this important question, we perform a sensitivity analysis of the threshold in \eqref{eq:threshold}. Our result is summarized in the following proposition.

\begin{proposition}\label{prop:sensitivity}
An increase in the vaccination coverage $v$ decreases the epidemic threshold $\bar{\tau}$ in \eqref{eq:threshold} if and only if
\begin{equation}\label{eq:sensitivity}
 (1-\gamma_{\mathrm{t}})[1-p_{\mathrm{q}}(1-\gamma_{\mathrm{q}})](1-\sigma_{\mathrm{v}})-(1-p_{\mathrm{q}})(1-\sigma_{\mathrm{n}})<0.
\end{equation}
\end{proposition}
\begin{proof}
The condition in \eqref{eq:sensitivity} is obtained by computing the partial derivative of $\bar \tau$ with respect to $v$, which is equal to the left-hand size of \eqref{eq:sensitivity}, multiplied by a positive quantity.
\end{proof}

From Proposition~\ref{prop:sensitivity}, we observe that the vaccination coverage has a non-trivial effect on the epidemic threshold. In particular, whether an increase in the vaccination coverage facilitates the prevention of an epidemic outbreak depends on the characteristics of the vaccine (i.e., its effectiveness against transmission and severe illness), the probability of developing severe illness, and the responsibility of vaccinated and non-vaccinated individuals. This is consistent with the observations made on a simpler model, in~\cite{frieswijk2022ecc}.

Finally, Figure~\ref{fig:threshold} reports some observations concerning the role of the homophily level $\theta$. Our numerical simulations show that high levels of homophily facilitate the spread of epidemic diseases, in particular when combined with lower levels of responsibility for the non-vaccinated individuals. This suggests that neglecting the polarization that can emerge during a pandemic, with clusters of individuals who disregard the use of protective measures and refuse to be vaccinated, may lead to a dangerous underestimation of the risk of a local outbreak.

To conclude this section, we observe that, while vaccination may surprisingly act as a double-edged sword in preventing epidemic outbreaks, sometimes calling for an increased testing effort, its effect on mitigating endemic diseases is more predictable. In fact, the Monte Carlo simulations reported in Fig.~\ref{fig:endemic} suggest that increasing vaccination coverage of a population is always beneficial (the number of infections always decreases if we increase the vaccination coverage $v$), excluding the specific scenario in which non-vaccinated individuals are significantly more responsible than their vaccinated peers. In this scenario, increasing the vaccination coverage may be harmful, as illustrated in Fig.~\ref{fig:endemic_responsability}.

\section{Numerical results}\label{sec:simulations}

\noindent In this section, we will expand on Remark~\ref{rem:network} employing numerical simulations on a backbone, while focusing on a case study motivated by the COVID-19 pandemic. 

In Remarks~\ref{rem:npi}--\ref{rem:sigma}, we discussed how our modeling framework is amenable to several generalizations. For some of them, our analytical findings can be readily extended. For instance, Theorem~\ref{theo1} can be extended to incorporate non-pharmaceutical interventions as suggested in Remark~\ref{rem:npi} by substituting $\lambda$ with $(1-\eta)\lambda$ in \eqref{eq:threshold}, where $\eta\in [0,1]$ is the effectiveness of non-pharmaceutical interventions. 
Other extensions, instead, increase the complexity of the dynamics, thereby hindering its analytical treatment and preventing a direct extension of the mean-field approach used to derive our theoretical findings. However, the implementation of our model, grounded in the activity-driven network formalism, enables its numerical treatment via fast Monte Carlo simulation campaigns. 

Here, following Remark~\ref{rem:network}, we extend our modeling framework by embedding it onto a network structure which constrains the interactions that an individual may establish. Inspired by activity-driven networks with a backbone~\cite{karsai2014backbone,surano2019backbone,Nadini2020}, we assume that contacts are constrained by an underlying time-invariant network $\mathcal H=(\mathcal V,\mathcal F)$, and we define the social contacts of $j$ as the neighbors of $j$ on the backbone, that is, $\mathcal N_j:=\{k:(j,k)\in\mathcal F\}$. The network formation process described in Section~\ref{sec:formation} is modified as follows. When the clock associated with individual $j$ ticks, then $j$ activates and has an interaction with another individual $k$. With probability $\theta$, $k$ is selected uniformly at random from the social contacts of $j$ that are in $j$'s subpopulation, $\mathcal N_j\cap \mathcal V_\alpha$ (where $\alpha\in\{\mathrm{v},\mathrm{n}\}$ is the vaccination status of $j$); whereas with probability $1-\theta$, $k$ is chosen uniformly at random from the entire set $\mathcal N_j$.

We consider a case study inspired by the ongoing COVID-19 pandemic and global vaccination campaign. Specifically, we utilize model parameters calibrated to reflect some characteristics of COVID-19 and of the situation in the Netherlands as of early November 2021, estimated in our previous work~\cite{frieswijk2022ecc} from clinical and epidemiological data~\cite{phucharoen2020characteristics,dagan2021bnt162b2}, and reported in Table~\ref{tab:covid}.

\begin{figure}
\vspace{0.25cm}
\centering
 \subfloat[Eradication probability]{
%
%
\begin{tikzpicture}

\begin{axis}[%
 axis lines=left,
   x  axis line style={->},
      y  axis line style={-},
width=\ll cm,
height=\h cm,
xlabel={testing rate, $\tau$},
ylabel={eradication},y tick label style={/pgf/number format/.cd,%
          scaled y ticks = false,
          set thousands separator={},
          fixed},
          x tick label style={/pgf/number format/.cd,%
          scaled y ticks = false,
          set thousands separator={},
          fixed},
scale only axis,
xmin=0,
xmax=0.195,
ymin=-.02,
ymax=1.02,
 extra y ticks ={0},
    extra y tick labels={$0$},
     extra x ticks ={0},
    extra x tick labels={$0$},
axis background/.style={fill=white}
]

\addplot [ultra thick, dashed, color=gray]
  table[row sep=crcr]{%
0.1087	0\\
0.1087	1\\
};

\addplot [ultra thick, color=cyan]
  table[row sep=crcr]{%
0	0\\
0.01	0\\
0.02	0\\
0.03	0\\
0.04	0\\
0.05	0\\
0.06	0\\
0.07	0\\
0.08	0\\
0.09	0\\
0.1	0\\
0.11	0\\
0.12	.1\\
0.13	.6\\
0.14	.9\\
0.15	1\\
0.16	1\\
0.17	1\\
0.18	1\\
0.19	1\\
0.2	1\\
};

\addplot [ultra thick, color=blue]
  table[row sep=crcr]{%
0	0\\
0.01	0\\
0.02	0\\
0.03	0\\
0.04	0\\
0.05	0\\
0.06	0\\
0.07	0\\
0.08	0\\
0.09	0\\
0.1	0\\
0.11	.3\\
0.12	.7\\
0.13	.9\\
0.14	1\\
0.15	1\\
0.16	1\\
0.17	1\\
0.18	1\\
0.19	1\\
0.2	1\\
};

\addplot [ultra thick, color=red]
  table[row sep=crcr]{%
0	0\\
0.01	0\\
0.02	0\\
0.03	0\\
0.04	0\\
0.05	0\\
0.06	0\\
0.07	0\\
0.08	0\\
0.09	0\\
0.1	0\\
0.11	.3\\
0.12	.6\\
0.13	1\\
0.14	1\\
0.15	1\\
0.16	1\\
0.17	1\\
0.18	1\\
0.19	1\\
0.2	1\\
};

\end{axis}

\end{tikzpicture}%
} \hspace{0.5cm} \subfloat[Steady-state infected fraction]{
%
%
\begin{tikzpicture}

\begin{axis}[%
 axis lines=left,
   x  axis line style={->},
      y  axis line style={->},
width=\ll cm,
height=\h cm,
xlabel={testing rate, $\tau$},
ylabel={infections},y tick label style={/pgf/number format/.cd,%
          scaled y ticks = false,
          set thousands separator={},
          fixed},
          x tick label style={/pgf/number format/.cd,%
          scaled y ticks = false,
          set thousands separator={},
          fixed},
scale only axis,
xmin=0,
xmax=0.195,
ymin=0,
ymax=0.059,
 extra y ticks ={0},
    extra y tick labels={$0$},
     extra x ticks ={0},
    extra x tick labels={$0$},
legend style={at={(1.05,1.08)},legend cell align=left, align=left, draw=none,fill=none, font=\footnotesize},
axis background/.style={fill=white}
]

\addplot [color=cyan, line width=1.5pt]
 plot [error bars/.cd, y dir = both, y explicit]
 table[row sep=crcr, y error plus index=2, y error minus index=3]{%
0	0.043799	0.00144643352228853	0.00144643352228853\\
0.01	0.038847	0.00166133674419125	0.00166133674419125\\
0.02	0.035025	0.00102642076752178	0.00102642076752178\\
0.03	0.03022	0.00124184214133681	0.00124184214133681\\
0.04	0.026133	0.00173160315130228	0.00173160315130228\\
0.05	0.022407	0.00140292878849926	0.00140292878849926\\
0.06	0.018632	0.00246584514550286	0.00246584514550286\\
0.07	0.015079	0.0019263808923471	0.0019263808923471\\
0.08	0.011354	0.00186163272167203	0.00186163272167203\\
0.09	0.008176	0.00295848609501549	0.00295848609501549\\
0.1	0.005263	0.0018072230990113	0.0018072230990113\\
0.11	0.002667	0.00209540500753434	0.00209540500753434\\
0.12	0.000985	0.00202439157674596	0.00202439157674596\\
0.13	0.000412	0.000782715147164024	0.000782715147164024\\
0.14	3e-05	0.0001764	0.0001764\\
0.15	0	0	0\\
0.16	0	0	0\\
0.17	0	0	0\\
0.18	0	0	0\\
0.19	0	0	0\\
0.2	0	0	0\\
};
\addlegendentry{BA}

\addplot [color=blue, line width=1.5pt]
 plot [error bars/.cd, y dir = both, y explicit]
 table[row sep=crcr, y error plus index=2, y error minus index=3]{%
0	0.046435	0.00151042768247937	0.00151042768247937\\
0.01	0.041145	0.00101807803630173	0.00101807803630173\\
0.02	0.035452	0.000982966869024588	0.000982966869024588\\
0.03	0.030998	0.00116043918306821	0.00116043918306821\\
0.04	0.025353	0.00163596140834679	0.00163596140834679\\
0.05	0.02082	0.00233855188695911	0.00233855188695911\\
0.06	0.01625	0.0012650126070518	0.0012650126070518\\
0.07	0.011191	0.00195248762311058	0.00195248762311058\\
0.08	0.007796	0.0022348160314442	0.0022348160314442\\
0.09	0.003328	0.00222783639471125	0.00222783639471125\\
0.1	0.001906	0.00160811281395305	0.00160811281395305\\
0.11	0.000396	0.000719039563862796	0.000719039563862796\\
0.12	0.000146	0.000534559233761797	0.000534559233761797\\
0.13	1.5e-05	8.82e-05	8.82e-05\\
0.14	2.5e-05	0.000147	0.000147\\
0.15	0	0	0\\
0.16	0	0	0\\
0.17	0	0	0\\
0.18	0	0	0\\
0.19	0	0	0\\
0.2	0	0	0\\
};
\addlegendentry{ER}

\addplot [color=red, line width=1.5pt]
 plot [error bars/.cd, y dir = both, y explicit]
 table[row sep=crcr, y error plus index=2, y error minus index=3]{%
0	0.046066	0.00127074525787036	0.00127074525787036\\
0.01	0.041592	0.00209174887823563	0.00209174887823563\\
0.02	0.036247	0.00215548410191307	0.00215548410191307\\
0.03	0.031044	0.00136931497267795	0.00136931497267795\\
0.04	0.025665	0.00132192519304233	0.00132192519304233\\
0.05	0.021322	0.00273208533278154	0.00273208533278154\\
0.06	0.016386	0.00193330330533002	0.00193330330533002\\
0.07	0.011634	0.00163967091527538	0.00163967091527538\\
0.08	0.007826	0.00163090845064951	0.00163090845064951\\
0.09	0.00461	0.00223374657112216	0.00223374657112216\\
0.1	0.001898	0.00221914589551926	0.00221914589551926\\
0.11	0.000512	0.00118328821577839	0.00118328821577839\\
0.12	7.1e-05	0.00028812	0.00028812\\
0.13	0	0	0\\
0.14	0	0	0\\
0.15	0	0	0\\
0.16	0	0	0\\
0.17	0	0	0\\
0.18	0	0	0\\
0.19	0	0	0\\
0.2	0	0	0\\
};
\addlegendentry{Complete}

\end{axis}
\end{tikzpicture}%
}
 \caption{Numerical estimation (averaged over $10$ independent runs) of (a) the eradication probability and (b) the steady-state fraction of infected individuals ($I_{\mathrm{v}}(t)+I_{\mathrm{n}}(t)+Q_{\mathrm{v}}(t)+Q_{\mathrm{n}}(t)$) at $T=200$, for increasing values of the control parameter $\tau$ and different backbone networks. The gray vertical dashed line is the threshold, computed analytically from Theorem~\ref{theo1}, which is exact for the complete network. Common parameters are $\sigma_{\mathrm{n}}=\theta=0.5$; the rest of the parameters are summarized in Table~\ref{tab:covid}.}
 \label{fig:network1}
\end{figure}
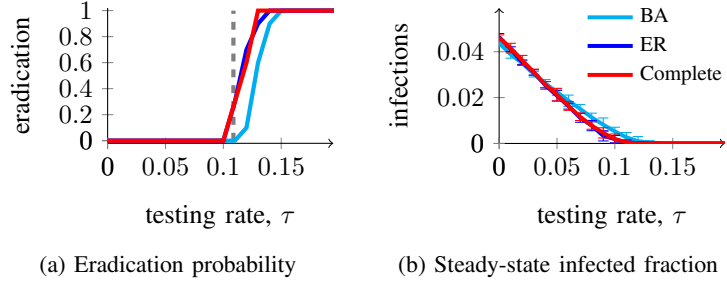
\begin{table}
\centering
\caption{Parameter values of the COVID-inspired case study.}\label{tab:covid}
\begin{tabular}{r|c|c|c|c|c|c|c}
 & $n$&$\lambda$&$\beta$&$v$&$p_{\text{q}}$&$\gamma_{\text{t}}$&$\gamma_{\text{q}}$\\
\hline
value & $10\,000$& $0.36$& $0.1$& $0.821$&$0.19$&$0.65$&$0.92$
\end{tabular}
\end{table}

First, we perform a set of Monte Carlo simulations to show that the threshold behavior, proved analytically in Section~\ref{section:mainresults} in the absence of backbone networks, is an inherent property of the epidemic model and is robust to the addition of a backbone of social contacts. To this aim, we numerically estimate the probability of eradicating the disease and the steady-state number of infections for different values of the testing rate $\tau$ and on three different network structures: an Erdős–Rényi random graph, where each undirected link is present with probability $p=0.01$; a Barab\'asi-Albert network, where degrees are power-law distributed, and a complete network, where our theoretical results from Theorem~\ref{theo1} can be directly applied. Our simulations, reported in Fig.~\ref{fig:network1}, suggest that threshold behavior is present for all types of networks. Moreover, the value of the threshold seems quite robust to the presence of a backbone network. However, a certain degree of shifting can be observed, calling for further numerical studies.

In view of these observations, we perform a more extensive set of simulations to shed light on how the presence of a network structure affects the epidemic threshold. 
In our simulations, we consider a population of $n=10\,000$ individuals, whose contacts occur on a backbone network generated as an Erdős–Rényi random graph, where each undirected link $(j,k)$ is present in $\mathcal F$ with probability $p=0.01$. In Fig.~\ref{fig:network}, we report the epidemic threshold and the long-term total fraction of infected individuals, for different values of homophily $\theta$ and responsibility of non-vaccinated individuals $\sigma_{\mathrm{n}}$. The threshold is estimated, similar to~\cite{Moinet2017pre,burstiness}, employing a Monte Carlo-based approach detailed in Appendix~\ref{app:numerics}.

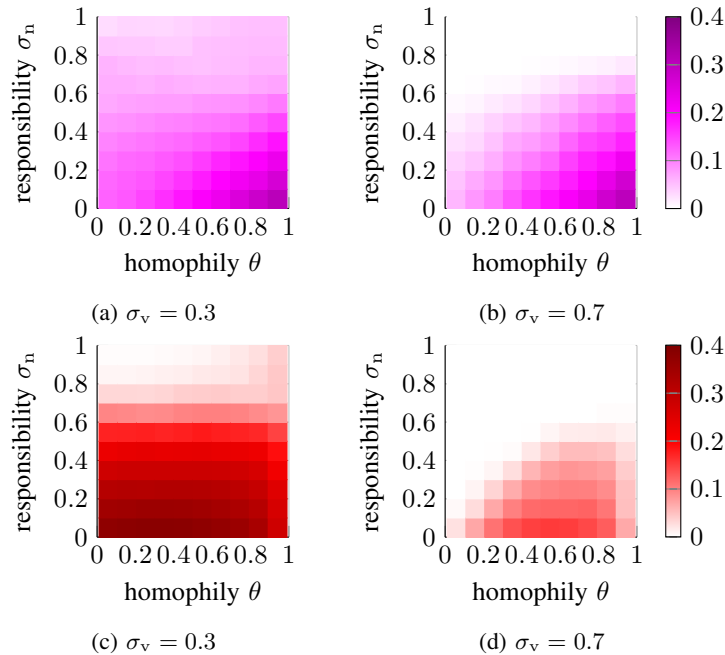
\begin{figure}
\vspace{0.25cm}
 \centering
  \subfloat[$\sigma_{\mathrm{v}}=0.3$]{\begin{tikzpicture}
\begin{axis}[%
colormap={mymap}{[1pt] rgb(0pt)=(1,1,1); rgb(127pt)=(1,0,1); rgb(255pt)=(.5,0,.5)},
width=\qq cm,
height=\qq cm,
scale only axis,
xmin=0,
xmax=1,
ymin=0,
point meta max=.4,
point meta min=0,
ymax=1,
axis background/.style={fill=white},
axis x line*=bottom,
axis y line*=left,
xlabel={homophily $\theta$},
ylabel={responsibility $\sigma_{\mathrm{n}}$},
xmajorgrids,
ymajorgrids,]

\addplot[%
surf,
shader=flat, draw=none, mesh/rows=11]
table[row sep=crcr, point meta=\thisrow{c}] {%
x	y	c\\
0.01	0	0.123272727272727\\
0.01	0.1	0.116\\
0.01	0.2	0.116\\
0.01	0.3	0.0898181818181818\\
0.01	0.4	0.0963636363636364\\
0.01	0.5	0.0701818181818182\\
0.01	0.6	0.0563636363636364\\
0.01	0.7	0.0563636363636364\\
0.01	0.8	0.0629090909090909\\
0.01	0.9	0.0301818181818182\\
0.01	1	0.012\\
0.108	0	0.129818181818182\\
0.108	0.1	0.123272727272727\\
0.108	0.2	0.123272727272727\\
0.108	0.3	0.109454545454545\\
0.108	0.4	0.0898181818181818\\
0.108	0.5	0.076\\
0.108	0.6	0.0694545454545455\\
0.108	0.7	0.0694545454545455\\
0.108	0.8	0.0498181818181818\\
0.108	0.9	0.036\\
0.108	1	0.03\\
0.206	0	0.136363636363636\\
0.206	0.1	0.136363636363636\\
0.206	0.2	0.123272727272727\\
0.206	0.3	0.116\\
0.206	0.4	0.0963636363636364\\
0.206	0.5	0.0832727272727273\\
0.206	0.6	0.0629090909090909\\
0.206	0.7	0.0498181818181818\\
0.206	0.8	0.0498181818181818\\
0.206	0.9	0.0367272727272727\\
0.206	1	0.03\\
0.304	0	0.169818181818182\\
0.304	0.1	0.142909090909091\\
0.304	0.2	0.129818181818182\\
0.304	0.3	0.116\\
0.304	0.4	0.102909090909091\\
0.304	0.5	0.0832727272727273\\
0.304	0.6	0.0701818181818182\\
0.304	0.7	0.0563636363636364\\
0.304	0.8	0.0498181818181818\\
0.304	0.9	0.0367272727272727\\
0.304	1	0.0236363636363636\\
0.402	0	0.182909090909091\\
0.402	0.1	0.156727272727273\\
0.402	0.2	0.150181818181818\\
0.402	0.3	0.123272727272727\\
0.402	0.4	0.109454545454545\\
0.402	0.5	0.0898181818181818\\
0.402	0.6	0.0636363636363636\\
0.402	0.7	0.0694545454545455\\
0.402	0.8	0.036\\
0.402	0.9	0.036\\
0.402	1	0.0301818181818182\\
0.5	0	0.209818181818182\\
0.5	0.1	0.182909090909091\\
0.5	0.2	0.149454545454545\\
0.5	0.3	0.129818181818182\\
0.5	0.4	0.109454545454545\\
0.5	0.5	0.0898181818181818\\
0.5	0.6	0.0629090909090909\\
0.5	0.7	0.0498181818181818\\
0.5	0.8	0.0432727272727273\\
0.5	0.9	0.0432727272727273\\
0.5	1	0.03\\
0.598	0	0.230181818181818\\
0.598	0.1	0.196\\
0.598	0.2	0.182909090909091\\
0.598	0.3	0.149454545454545\\
0.598	0.4	0.116\\
0.598	0.5	0.102909090909091\\
0.598	0.6	0.076\\
0.598	0.7	0.0498181818181818\\
0.598	0.8	0.0498181818181818\\
0.598	0.9	0.0563636363636364\\
0.598	1	0.03\\
0.696	0	0.256363636363636\\
0.696	0.1	0.222909090909091\\
0.696	0.2	0.196\\
0.696	0.3	0.156\\
0.696	0.4	0.116\\
0.696	0.5	0.0963636363636364\\
0.696	0.6	0.0629090909090909\\
0.696	0.7	0.0563636363636364\\
0.696	0.8	0.0498181818181818\\
0.696	0.9	0.0432727272727273\\
0.696	1	0.0432727272727273\\
0.794	0	0.302909090909091\\
0.794	0.1	0.243272727272727\\
0.794	0.2	0.209818181818182\\
0.794	0.3	0.176363636363636\\
0.794	0.4	0.136363636363636\\
0.794	0.5	0.109454545454545\\
0.794	0.6	0.076\\
0.794	0.7	0.0498181818181818\\
0.794	0.8	0.0563636363636364\\
0.794	0.9	0.0563636363636364\\
0.794	1	0.0498181818181818\\
0.892	0	0.316\\
0.892	0.1	0.262909090909091\\
0.892	0.2	0.229454545454545\\
0.892	0.3	0.196\\
0.892	0.4	0.156\\
0.892	0.5	0.116\\
0.892	0.6	0.0767272727272727\\
0.892	0.7	0.0694545454545455\\
0.892	0.8	0.0563636363636364\\
0.892	0.9	0.0432727272727273\\
0.892	1	0.0498181818181818\\
0.99	0	0.342909090909091\\
0.99	0.1	0.309454545454546\\
0.99	0.2	0.262909090909091\\
0.99	0.3	0.222909090909091\\
0.99	0.4	0.169818181818182\\
0.99	0.5	0.136363636363636\\
0.99	0.6	0.0898181818181818\\
0.99	0.7	0.0498181818181818\\
0.99	0.8	0.0563636363636364\\
0.99	0.9	0.0563636363636364\\
0.99	1	0.0498181818181818\\
};

\end{axis}

\end{tikzpicture}%
\label{fig:5c}} \hspace{0.5cm}
  \subfloat[$\sigma_{\mathrm{v}}=0.7$]{\begin{tikzpicture}
\begin{axis}[%
colormap={mymap}{[1pt] rgb(0pt)=(1,1,1); rgb(127pt)=(1,0,1); rgb(255pt)=(.5,0,.5)},
width=\qq cm,
height=\qq cm,
scale only axis,
xmin=0,
xmax=1,
ymin=0,
point meta max=.4,
point meta min=0,
ymax=1,
axis background/.style={fill=white},
axis x line*=bottom,
axis y line*=left,
xlabel={homophily $\theta$},
ylabel={responsibility $\sigma_{\mathrm{n}}$},
xmajorgrids,
ymajorgrids,colorbar,
colorbar style={at={(1.15,1)}, width=.2cm, yticklabel style={
        /pgf/number format/fixed,
        /pgf/number format/precision=5
},},
]

\addplot[%
surf,
shader=flat, draw=none, mesh/rows=11]
table[row sep=crcr, point meta=\thisrow{c}] {%
x	y	c\\
0.01	0	0.0432727272727273\\
0.01	0.1	0.0367272727272727\\
0.01	0.2	0.03\\
0.01	0.3	0.03\\
0.01	0.4	0.012\\
0.01	0.5	0.006\\
0.01	0.6	0\\
0.01	0.7	0\\
0.01	0.8	0\\
0.01	0.9	0\\
0.01	1	0\\
0.108	0	0.0767272727272727\\
0.108	0.1	0.0694545454545455\\
0.108	0.2	0.0498181818181818\\
0.108	0.3	0.03\\
0.108	0.4	0.024\\
0.108	0.5	0.012\\
0.108	0.6	0\\
0.108	0.7	0\\
0.108	0.8	0\\
0.108	0.9	0\\
0.108	1	0\\
0.206	0	0.102909090909091\\
0.206	0.1	0.0767272727272727\\
0.206	0.2	0.0694545454545455\\
0.206	0.3	0.036\\
0.206	0.4	0.036\\
0.206	0.5	0.018\\
0.206	0.6	0.006\\
0.206	0.7	0\\
0.206	0.8	0\\
0.206	0.9	0\\
0.206	1	0\\
0.304	0	0.123272727272727\\
0.304	0.1	0.103636363636364\\
0.304	0.2	0.0832727272727273\\
0.304	0.3	0.0636363636363636\\
0.304	0.4	0.0563636363636364\\
0.304	0.5	0.03\\
0.304	0.6	0.012\\
0.304	0.7	0\\
0.304	0.8	0\\
0.304	0.9	0\\
0.304	1	0\\
0.402	0	0.156\\
0.402	0.1	0.149454545454545\\
0.402	0.2	0.103636363636364\\
0.402	0.3	0.0898181818181818\\
0.402	0.4	0.0629090909090909\\
0.402	0.5	0.036\\
0.402	0.6	0.012\\
0.402	0.7	0\\
0.402	0.8	0\\
0.402	0.9	0\\
0.402	1	0\\
0.5	0	0.182909090909091\\
0.5	0.1	0.149454545454545\\
0.5	0.2	0.136363636363636\\
0.5	0.3	0.076\\
0.5	0.4	0.0767272727272727\\
0.5	0.5	0.0498181818181818\\
0.5	0.6	0.03\\
0.5	0.7	0.006\\
0.5	0.8	0\\
0.5	0.9	0\\
0.5	1	0\\
0.598	0	0.217090909090909\\
0.598	0.1	0.182909090909091\\
0.598	0.2	0.150181818181818\\
0.598	0.3	0.116\\
0.598	0.4	0.0963636363636364\\
0.598	0.5	0.0694545454545455\\
0.598	0.6	0.036\\
0.598	0.7	0.006\\
0.598	0.8	0\\
0.598	0.9	0\\
0.598	1	0\\
0.696	0	0.236\\
0.696	0.1	0.196\\
0.696	0.2	0.177090909090909\\
0.696	0.3	0.156727272727273\\
0.696	0.4	0.116\\
0.696	0.5	0.0832727272727273\\
0.696	0.6	0.0563636363636364\\
0.696	0.7	0.018\\
0.696	0.8	0\\
0.696	0.9	0\\
0.696	1	0\\
0.794	0	0.269454545454546\\
0.794	0.1	0.236\\
0.794	0.2	0.196\\
0.794	0.3	0.163272727272727\\
0.794	0.4	0.136363636363636\\
0.794	0.5	0.102909090909091\\
0.794	0.6	0.0701818181818182\\
0.794	0.7	0.018\\
0.794	0.8	0\\
0.794	0.9	0\\
0.794	1	0\\
0.892	0	0.316\\
0.892	0.1	0.276\\
0.892	0.2	0.236\\
0.892	0.3	0.156\\
0.892	0.4	0.156\\
0.892	0.5	0.109454545454545\\
0.892	0.6	0.0701818181818182\\
0.892	0.7	0.03\\
0.892	0.8	0\\
0.892	0.9	0\\
0.892	1	0\\
0.99	0	0.316\\
0.99	0.1	0.309454545454546\\
0.99	0.2	0.262909090909091\\
0.99	0.3	0.229454545454545\\
0.99	0.4	0.156\\
0.99	0.5	0.142909090909091\\
0.99	0.6	0.0898181818181818\\
0.99	0.7	0.0432727272727273\\
0.99	0.8	0\\
0.99	0.9	0\\
0.99	1	0\\
};

\end{axis}

\end{tikzpicture}%
\label{fig:5d}}\\
  \subfloat[$\sigma_{\mathrm{v}}=0.3$]{\begin{tikzpicture}
\begin{axis}[%
colormap={mymap}{[1pt] rgb(0pt)=(1,1,1); rgb(127pt)=(1,0,0); rgb(255pt)=(.5,0,0)},
width=\qq cm,
height=\qq cm,
scale only axis,
xmin=0,
xmax=1,
ymin=0,
point meta max=.3,
point meta min=0,
ymax=1,
axis background/.style={fill=white},
axis x line*=bottom,
axis y line*=left,
xlabel={homophily $\theta$},
ylabel={responsibility $\sigma_{\mathrm{n}}$},
xmajorgrids,
ymajorgrids,]

\addplot[%
surf,
shader=flat, draw=none, mesh/rows=11]
table[row sep=crcr, point meta=\thisrow{c}] {%
x	y	c\\
0.01	0	0.28921\\
0.01	0.1	0.2718\\
0.01	0.2	0.24514\\
0.01	0.3	0.22149\\
0.01	0.4	0.18824\\
0.01	0.5	0.15571\\
0.01	0.6	0.10154\\
0.01	0.7	0.03549\\
0.01	0.8	0.01018\\
0.01	0.9	0.00127\\
0.01	1	0.00088\\
0.108	0	0.29264\\
0.108	0.1	0.27154\\
0.108	0.2	0.2504\\
0.108	0.3	0.23053\\
0.108	0.4	0.18927\\
0.108	0.5	0.15298\\
0.108	0.6	0.10746\\
0.108	0.7	0.03995\\
0.108	0.8	0.00896\\
0.108	0.9	0.0038\\
0.108	1	0.00052\\
0.206	0	0.29932\\
0.206	0.1	0.27709\\
0.206	0.2	0.25027\\
0.206	0.3	0.22518\\
0.206	0.4	0.19816\\
0.206	0.5	0.16609\\
0.206	0.6	0.09436\\
0.206	0.7	0.0363\\
0.206	0.8	0.0091\\
0.206	0.9	0.00161\\
0.206	1	0.0012\\
0.304	0	0.29292\\
0.304	0.1	0.27991\\
0.304	0.2	0.25551\\
0.304	0.3	0.22294\\
0.304	0.4	0.19548\\
0.304	0.5	0.15781\\
0.304	0.6	0.1076\\
0.304	0.7	0.03099\\
0.304	0.8	0.01413\\
0.304	0.9	0.00369\\
0.304	1	0.00113\\
0.402	0	0.29708\\
0.402	0.1	0.27772\\
0.402	0.2	0.25319\\
0.402	0.3	0.2288\\
0.402	0.4	0.19613\\
0.402	0.5	0.1553\\
0.402	0.6	0.0945\\
0.402	0.7	0.04131\\
0.402	0.8	0.01298\\
0.402	0.9	0.00493\\
0.402	1	0.00226\\
0.5	0	0.29645\\
0.5	0.1	0.27392\\
0.5	0.2	0.25403\\
0.5	0.3	0.22432\\
0.5	0.4	0.19374\\
0.5	0.5	0.16719\\
0.5	0.6	0.10905\\
0.5	0.7	0.0452\\
0.5	0.8	0.01459\\
0.5	0.9	0.00733\\
0.5	1	0.00232\\
0.598	0	0.28746\\
0.598	0.1	0.26992\\
0.598	0.2	0.24897\\
0.598	0.3	0.22824\\
0.598	0.4	0.19712\\
0.598	0.5	0.15607\\
0.598	0.6	0.10646\\
0.598	0.7	0.03764\\
0.598	0.8	0.02231\\
0.598	0.9	0.00797\\
0.598	1	0.00606\\
0.696	0	0.2829\\
0.696	0.1	0.2607\\
0.696	0.2	0.24437\\
0.696	0.3	0.22174\\
0.696	0.4	0.19239\\
0.696	0.5	0.15802\\
0.696	0.6	0.10647\\
0.696	0.7	0.04858\\
0.696	0.8	0.0151\\
0.696	0.9	0.01123\\
0.696	1	0.01361\\
0.794	0	0.27286\\
0.794	0.1	0.25406\\
0.794	0.2	0.23243\\
0.794	0.3	0.21563\\
0.794	0.4	0.19289\\
0.794	0.5	0.15286\\
0.794	0.6	0.09993\\
0.794	0.7	0.03439\\
0.794	0.8	0.01918\\
0.794	0.9	0.01954\\
0.794	1	0.01084\\
0.892	0	0.24679\\
0.892	0.1	0.2338\\
0.892	0.2	0.22012\\
0.892	0.3	0.20406\\
0.892	0.4	0.18048\\
0.892	0.5	0.15207\\
0.892	0.6	0.10219\\
0.892	0.7	0.0338\\
0.892	0.8	0.03005\\
0.892	0.9	0.02698\\
0.892	1	0.01743\\
0.99	0	0.18101\\
0.99	0.1	0.16691\\
0.99	0.2	0.17216\\
0.99	0.3	0.15403\\
0.99	0.4	0.12345\\
0.99	0.5	0.11125\\
0.99	0.6	0.07803\\
0.99	0.7	0.03394\\
0.99	0.8	0.03536\\
0.99	0.9	0.03583\\
0.99	1	0.03412\\
};

\end{axis}

\end{tikzpicture}%
\label{fig:5c_inf}} \hspace{0.5cm}
  \subfloat[$\sigma_{\mathrm{v}}=0.7$]{\begin{tikzpicture}
\begin{axis}[%
colormap={mymap}{[1pt] rgb(0pt)=(1,1,1); rgb(127pt)=(1,0,0); rgb(255pt)=(.5,0,0)},
width=\qq cm,
height=\qq cm,
scale only axis,
xmin=0,
xmax=1,
ymin=0,
point meta max=.4,
point meta min=0,
ymax=1,
axis background/.style={fill=white},
axis x line*=bottom,
axis y line*=left,
xlabel={homophily $\theta$},
ylabel={responsibility $\sigma_{\mathrm{n}}$},
xmajorgrids,
ymajorgrids,colorbar,
colorbar style={at={(1.15,1)}, width=.2cm, yticklabel style={
        /pgf/number format/fixed,
        /pgf/number format/precision=5
},},
]

\addplot[%
surf,
shader=flat, draw=none, mesh/rows=11]
table[row sep=crcr, point meta=\thisrow{c}] {%
x	y	c\\
0.01	0	0.00877\\
0.01	0.1	0.00119\\
0.01	0.2	0\\
0.01	0.3	0\\
0.01	0.4	2e-05\\
0.01	0.5	0\\
0.01	0.6	0\\
0.01	0.7	0\\
0.01	0.8	0\\
0.01	0.9	0\\
0.01	1	0\\
0.108	0	0.06619\\
0.108	0.1	0.01656\\
0.108	0.2	0.00048\\
0.108	0.3	0\\
0.108	0.4	2e-05\\
0.108	0.5	0\\
0.108	0.6	0\\
0.108	0.7	0\\
0.108	0.8	0\\
0.108	0.9	0\\
0.108	1	0\\
0.206	0	0.12599\\
0.206	0.1	0.06489\\
0.206	0.2	0.02866\\
0.206	0.3	0.00623\\
0.206	0.4	0.00099\\
0.206	0.5	0\\
0.206	0.6	0\\
0.206	0.7	0\\
0.206	0.8	0\\
0.206	0.9	0\\
0.206	1	0\\
0.304	0	0.15116\\
0.304	0.1	0.09965\\
0.304	0.2	0.07801\\
0.304	0.3	0.02277\\
0.304	0.4	0.00142\\
0.304	0.5	0.00019\\
0.304	0.6	0\\
0.304	0.7	0\\
0.304	0.8	0\\
0.304	0.9	0\\
0.304	1	0\\
0.402	0	0.16283\\
0.402	0.1	0.12945\\
0.402	0.2	0.08991\\
0.402	0.3	0.06714\\
0.402	0.4	0.01219\\
0.402	0.5	0\\
0.402	0.6	0\\
0.402	0.7	0\\
0.402	0.8	0\\
0.402	0.9	0\\
0.402	1	0\\
0.5	0	0.17077\\
0.5	0.1	0.13349\\
0.5	0.2	0.11177\\
0.5	0.3	0.09129\\
0.5	0.4	0.05936\\
0.5	0.5	0.0039\\
0.5	0.6	2e-05\\
0.5	0.7	0\\
0.5	0.8	0\\
0.5	0.9	0\\
0.5	1	0\\
0.598	0	0.17241\\
0.598	0.1	0.1394\\
0.598	0.2	0.09252\\
0.598	0.3	0.08822\\
0.598	0.4	0.06344\\
0.598	0.5	0.0281\\
0.598	0.6	0.00026\\
0.598	0.7	0\\
0.598	0.8	0\\
0.598	0.9	0\\
0.598	1	0\\
0.696	0	0.17371\\
0.696	0.1	0.12192\\
0.696	0.2	0.12402\\
0.696	0.3	0.09616\\
0.696	0.4	0.08406\\
0.696	0.5	0.03274\\
0.696	0.6	0.00025\\
0.696	0.7	0\\
0.696	0.8	0\\
0.696	0.9	0\\
0.696	1	0\\
0.794	0	0.14381\\
0.794	0.1	0.13304\\
0.794	0.2	0.10696\\
0.794	0.3	0.07127\\
0.794	0.4	0.06628\\
0.794	0.5	0.02704\\
0.794	0.6	0.00201\\
0.794	0.7	0\\
0.794	0.8	0\\
0.794	0.9	0\\
0.794	1	0\\
0.892	0	0.13309\\
0.892	0.1	0.09738\\
0.892	0.2	0.09038\\
0.892	0.3	0.09095\\
0.892	0.4	0.06812\\
0.892	0.5	0.03134\\
0.892	0.6	0.00955\\
0.892	0.7	0\\
0.892	0.8	0\\
0.892	0.9	0\\
0.892	1	0\\
0.99	0	0.03909\\
0.99	0.1	0\\
0.99	0.2	0.01149\\
0.99	0.3	0\\
0.99	0.4	0\\
0.99	0.5	0.00775\\
0.99	0.6	0\\
0.99	0.7	0\\
0.99	0.8	0\\
0.99	0.9	0\\
0.99	1	0\\
};

\end{axis}

\end{tikzpicture}%
\label{fig:5d_inf}}
 \caption{(a,b) The epidemic threshold, estimated numerically for the COVID-inspired case study with a backbone network; and (b,c) the total fraction of infected individuals ($I_{\mathrm{v}}(t)+I_{\mathrm{n}}(t)+Q_{\mathrm{v}}(t)+Q_{\mathrm{n}}(t)$) at $T=200$, for different values of the model parameters. Common parameters are summarized in Table~\ref{tab:covid}.}
 \label{fig:network}
\end{figure}

Our simulation results in Fig.~\ref{fig:network} suggest that i) responsibility of individuals is crucial: for low levels of responsibility, it is impossible to eradicate the disease without resorting to massive testing campaigns (Fig.~\ref{fig:network}c); ii) the role of homophily, already important in the absence of backbone networks, becomes even more critical, in particular when vaccinated individuals have a higher responsibility level than non-vaccinated ones (Fig.~\ref{fig:network}b,d). This can, e.g., be the case when non-vaccinated individuals belong to a minority of conspiracy theorists, since COVID-19-related conspiracy belief is negatively correlated with the willingness to vaccinate and infection-preventive behavior~\cite{conspiracyandvaccinewillingnessplusbehaviour}; iii) The responsibility level of vaccinated individuals seems to have a moderate role on the epidemic threshold. It has some effect only if non-vaccinated people are behaving responsibly and homophily is moderate (Fig.~\ref{fig:network}a,b). 
When above the threshold, however, we can observe by comparing Fig.~\ref{fig:network}c and Fig.~\ref{fig:network}d that the responsibility of vaccinated individuals plays a paramount role in reducing the fraction of infected individuals.

\section{Conclusion}\label{Sec:conclusion}
\noindent We proposed a polarized temporal network model for the spread of recurrent epidemic infections and investigated the effect of homophily, human behavior, and vaccination campaigns on the infection prevalence and the control of local outbreaks. Via a mean-field approach, we analytically derived the epidemic threshold. To include preferential contacts in the framework, we ran numerical simulations while confining interactions to an underlying network structure. For the simulations, the parameter values were calibrated to a case study on COVID-19. Our results suggest that, while vaccination is a very powerful measure to mitigate the number of deaths and the pressure on hospitals, its effect on the control of local outbreaks is nontrivial. Dependent on the characteristics of both the infection and the vaccine in question, vaccination campaigns might act as double-edged swords; they cut down the pressure on hospitals, but possibly also impede the control of local outbreaks. Analytical and numerical findings suggest that, in these scenarios, two paths are possible to completely eradicate the disease: either by relying on the population's responsibility or by resorting to testing on a massive scale. Furthermore, our simulations showed that a polarized network structure with a high degree of homophily presents a critical barricade in the control of local outbreaks.

Despite the generality of our modeling framework, some limitations should be highlighted. In particular, here we assumed that the decisions of the individuals on whether to vaccinate are fixed and made a priori. This could make sense for some infections, e.g., influenza viruses, but it neglects to consider scenarios in which vaccination decisions dynamically change. Moreover, we considered a situation in which individuals who are in favor of vaccination have already been vaccinated. Our model could be extended by considering changing views over time, following ~\cite{bauch2004}, for instance, and by introducing vaccine administration during an epidemic outbreak, as in~\cite{parino2021}. 

Besides incorporating these features into the model, several other research directions may be highlighted. First, our numerical results suggest that the threshold phenomenon proved analytically in the assumption of a complete backbone network, is an inherent feature of the epidemic model. Future efforts should be placed toward extending our analytical results. Second, one could extend the analytical treatment of the system beyond the stability of the disease-free equilibrium, by studying the behavior of the system above the epidemic threshold and by estimating the endemic equilibrium, similar to~\cite{frieswijk2022ecc}. Third, the model should be further extended by including extra compartments to capture, e.g., latency periods and (temporary) immunity. Finally, we plan to introduce a game-theoretic decision-making process in the model to capture the process through which mildly symptomatic individuals decide whether to maintain physical distance from others or not, as a function of the epidemic spreading and, potentially, other external factors~\cite{ye2021game,frieswijk2022mean}.

\bibliographystyle{ieeetr}
\bibliography{bib}


\appendices

\section{Proof of Theorem~\ref{theo1}}\label{prooftheo1}
\noindent Observe that the DFE of \eqref{y} is the state $( y_{\mathrm{n},\mathrm{s}} \ y_{\mathrm{n},\mathrm{i}} \ y_{\mathrm{n},\mathrm{q}} \ y_{\mathrm{v},\mathrm{s}} \ y_{\mathrm{v},\mathrm{i}} \ y_{\mathrm{v},\mathrm{q}} ) = (1-v,0,0,v, 0, 0)$, which is an equilibrium, since the right-hand side of \eqref{y} is equal to zero in this state. To study the local stability of the DFE, we recall that only four of the six equations of system \eqref{y} are linearly independent (Remark~\ref{rem:4}). In our analysis, we reduce the system to a 4-dimensional system, by choosing the macroscopic variables $y_{\mathrm{n},\mathrm{i}}$, $y_{\mathrm{n},\mathrm{q}}$, $y_{\mathrm{v},\mathrm{i}}$, and $y_{\mathrm{v},\mathrm{q}}$, and subsequently linearize \eqref{y} around the DFE of the original system (which coincides with the origin of the 4-dimensional reduced one), yielding
\begin{align}\label{eq:macro}
 \dot{y}_{\mathrm{n},\mathrm{i}} \hspace{1pt} = & \left[ 2  \lambda (1-p_{\mathrm{q}})  [ \theta + (1 -\theta)(1-v)](1-\sigma_{\mathrm{n}}) - \beta - \tau \right] y_{\mathrm{n},\mathrm{i}}  + 2 \lambda (1-p_{\mathrm{q}}) (1-\theta)(1-v)(1-\sigma_{\mathrm{v}}) y_{\mathrm{v},\mathrm{i}} , \notag\\
 \dot{y}_{\mathrm{v},\mathrm{i}} \hspace{1pt} = & 2 \lambda (1-\gamma_{\mathrm{t}}) [1-p_{\mathrm{q}}(1-\gamma_{\mathrm{q}})] v(1-\theta)(1-\sigma_{\mathrm{n}}) y_{\mathrm{n},\mathrm{i}}  + \left[ 2  \lambda (1-\gamma_{\mathrm{t}}) [1-p_{\mathrm{q}}(1-\gamma_{\mathrm{q}})]  [ \theta + v(1 -\theta) ](1-\sigma_{\mathrm{v}}) -\beta - \tau \right] y_{\mathrm{v},\mathrm{i}} \notag \\
 \dot{y}_{\mathrm{n},\mathrm{q}} = & \left[ 2  \lambda p_{\mathrm{q}}  [ \theta + (1 -\theta)(1-v)](1-\sigma_{\mathrm{n}})  + \tau \right] y_{\mathrm{n},\mathrm{i}}  +  2\lambda p_{\mathrm{q}} (1-\theta)(1-v)(1-\sigma_{\mathrm{v}}) y_{\mathrm{v},\mathrm{i}}  - \beta y_{\mathrm{n},\mathrm{q}}, \\
 \dot{y}_{\mathrm{v},\mathrm{q}} \hspace{1pt}= & 2 \lambda (1-\gamma_{\mathrm{t}}) p_{\mathrm{q}}(1-\gamma_{\mathrm{q}}) v (1-\theta)(1-\sigma_{\mathrm{n}}) y_{\mathrm{n},\mathrm{i}}  + \left[2  \lambda (1-\gamma_{\mathrm{t}}) p_{\mathrm{q}}(1-\gamma_{\mathrm{q}})  [ \theta + v(1 -\theta) ](1-\sigma_{\mathrm{v}})+ \tau  \right] y_{\mathrm{v},\mathrm{i}} - \beta y_{\mathrm{v},\mathrm{q}}. \notag
\end{align}
According to standard system-theoretic methods~\cite{rugh1996linear}, the (local) stability of the DFE is fully determined by the eigenvalues of the Jacobian matrix of~\eqref{eq:macro} evaluated at the origin. After re-sorting the equations in the order $y_{\mathrm{n},\mathrm{q}}$, $y_{\mathrm{v},\mathrm{q}}$, $y_{\mathrm{n},\mathrm{i}}$, $y_{\mathrm{v},\mathrm{i}}$, we observe that the Jacobian of \eqref{eq:macro} has the following block-triangular structure:
\begin{equation}\label{eq:jacobian}
 \begin{bNiceMatrix}[first-row,first-col]
 & y_{\mathrm{n},\mathrm{q}} & y_{\mathrm{v},\mathrm{q}} & y_{\mathrm{n},\mathrm{i}} & y_{\mathrm{v},\mathrm{i}} \\
y_{\mathrm{n},\mathrm{q}} & \ast & 0 & \ast &\ast\\
y_{\mathrm{v},\mathrm{q}} &0 & \ast & \ast &*\\
y_{\mathrm{n},\mathrm{i}} &0 & 0 & \ast &\ast\\
y_{\mathrm{v},\mathrm{i}} & 0 & 0 & \ast &\ast\\
\end{bNiceMatrix},
\end{equation}
where an asterisk ($\ast$) is used to denote a nonzero entry. 
The block ($y_{\mathrm{n},\mathrm{q}},y_{\mathrm{v},\mathrm{q}}$) is diagonal and associated with eigenvalue $-\beta<0$, with multiplicity $2$. Through a direct computation, we establish that the eigenvalues of the block ($y_{\mathrm{n},\mathrm{i}},y_{\mathrm{v},\mathrm{i}}$) are given by
\begin{align}\label{eigenvalue}
 & \lambda \xi - \beta - \tau \pm \lambda \sqrt{ \xi^2 - 4 \theta (1-\gamma_{\mathrm{t}})[1-p_{\mathrm{q}}][1-p_{\mathrm{q}}(1-\gamma_{\mathrm{q}})] (1-\sigma_{\mathrm{n}})(1-\sigma_{\mathrm{v}}) },
\end{align}
with $\xi$ defined as in \eqref{xi}. Observe that 
\begin{align}\label{polynomial}
 \xi^2 & - 4 \theta (1-\gamma_{\mathrm{t}})[1-p_{\mathrm{q}}][1-p_{\mathrm{q}}(1-\gamma_{\mathrm{q}})] (1-\sigma_{\mathrm{n}})(1-\sigma_{\mathrm{v}}) \notag\\
& \hspace{2cm} = \theta^2 \left[ v \rho +(1-v) \phi\right]^2 + 2\theta \left[ v(1-v)(\phi - \rho )^2 - \phi \rho \right] + \left[(1-v) \rho +v \phi \right]^2, 
\end{align}
where the right-hand side is a polynomial in $\theta$, and where $\phi : = (1-\gamma_{\mathrm{t}})[1-p_{\mathrm{q}}(1-\gamma_{\mathrm{q}})](1-\sigma_{\mathrm{v}})$ and $\rho : = [1-p_{\mathrm{q}}](1-\sigma_{\mathrm{n}})$. The roots of \eqref{polynomial} are complex and are given by 
$$ \frac{-v(1-v) (\phi- \rho)^2 + \phi \rho \pm 2 i \sqrt{v(1-v) \phi \rho (\phi - \rho)^2}}{(v \rho +(1-v)\phi)^2}.$$ Since the leading coefficient of \eqref{polynomial} is positive, and the roots of \eqref{polynomial} are complex, it follows that \eqref{polynomial} is strictly positive for all $\theta$. Thus, the two eigenvalues in \eqref{eigenvalue} are real, and the DFE is locally asymptotically stable if and only if the maximum eigenvalue (i.e., the one with the positive sign in \eqref{eigenvalue}) is negative, which is achieved if $\tau > \bar{\tau}$. On the contrary, if $\tau < \bar{\tau}$, the maximum eigenvalue of the Laplacian is positive, and the DFE is unstable.\qed

\section{Proof of Proposition~\ref{prop:monotonicity}}\label{app:monotonicity}
\noindent All the statements are derived from observing monotonicity properties of~\eqref{eq:macro} with respect to the considered parameters. For $p_{\mathrm{q}}$ and $\gamma_{\mathrm{t}}$, monotonicity holds only for the equations $\dot y_{\mathrm{n},\mathrm{i}}$ and $\dot y_{\mathrm{v},\mathrm{i}}$, which, however, are the two that determine the stability of the DFE due to the block-triangular structure of the Jacobian in \eqref{eq:jacobian}.
\qed

\section{Details on numerical simulations}\label{app:numerics}
\noindent The epidemic threshold is estimated as follows. We set a range of values for the parameter $\tau$. For each value of $\tau$, we initialize the epidemics with $10$ infected individuals, and we estimate the probability that the disease is extinguished within a fixed time-horizon (we set it to $T=200$) through $10$ independent Monte Carlo simulations. Following~\cite{Moinet2017pre,burstiness}, the estimation of the threshold is the value of $\tau$ that maximizes the standard deviation of the eradication probability. Because the output of each simulation is a binary variable ('$0$' for eradication, `$1$' for endemicity), the value of $\tau$ that maximizes the standard deviation coincides with the value with estimated eradication probability close to $0.5$. To optimize the process, we adopt a two-step procedure. First, we consider a wide range of values of $\tau$ with a large step size $\Delta \tau$, and we estimate the eradication probability starting from a high value of $\tau$ and subsequently decrease it, until we reach a value $\tilde{\tau}$ with an estimated eradication probability smaller than $0.5$. Secondly, we use the proposed algorithm to estimate the threshold in the range $\tau\in[\tilde\tau-\Delta\tau,\tilde\tau+\Delta\tau]$, while choosing a smaller step size $\Delta\tau$. The code used for all our simulations is freely available at \url{https://github.com/lzino90/vaccine_siq}.

\end{document}